\documentclass[10pt,reqno]{amsart}
\usepackage{amsfonts}
\usepackage{mathrsfs}  
\usepackage{amsthm}
\usepackage{amsxtra}
\usepackage{amssymb}
\usepackage{tikz}
\usepackage{caption}
\usepackage{graphicx}
\usepackage{subfigure}
\usepackage{multirow,booktabs}
\usepackage{float} 
\usepackage[pagewise]{lineno}

\newtheorem{theorem}{Theorem}[section]
\newtheorem{cor}[theorem]{Corollary}
\newtheorem{defi}[theorem]{Definition}

\newtheorem{proposition}[theorem]{Proposition}
\newtheorem{assumption}[theorem]{Assumption}
\newtheorem{remark}[theorem]{Remark}

\newcommand*{\dif}{\mathop{}\!\mathrm{d}}

\newcommand{\R}{\mathbb{R}}

\title[Photography transforms and their inversion formulas]{The photography transforms and their analytic inversion formulas}
\author{Duo Liu}
\address{School of Mathematics and Statistics, Beijing Jiaotong University, Beijing 100044, China}
\email{20118001@bjtu.edu.cn}
\author{Gangrong Qu}
\address{School of Mathematics and Statistics, Beijing Jiaotong University, Beijing 100044, China}
\email{grqu@bjtu.edu.cn}
\author{Shan Gao}
\address{School of Statistics and Data Science, Beijing Wuzi University, Beijing 101149, China}
\email{gaoshan@bwu.edu.cn}

\begin{document}
\begin{abstract}
The light field reconstruction from the focal stack can be mathematically formulated as an ill-posed integral equation inversion problem. Although the previous research about this problem has made progress both in practice and theory, its forward problem and inversion in a general form still need to be studied. In this paper, to model the forward problem rigorously, we propose three types of photography transforms with different integral geometry characteristics that extend the forward operator to the arbitrary $n$-dimensional case. We prove that these photography transforms are equivalent to the Radon transform with the coupling relation between variables. We also obtain some properties of the photography transforms, including the Fourier slice theorem, the convolution theorem, and the convolution property of the dual operator, which are very similar to those of the classic Radon transform. Furthermore, the representation of the normal operator and the analytic inversion formula for the photography transforms are derived and they are quite different from those of the classic Radon transform.
\end{abstract}

	

\maketitle

\section{\bf Introduction}
A light field is modeled as a function that describes the radiance of all light rays coming from a position along a particular direction in space \cite{GA}. Since the light field records rich high-dimensional visual information of scene \cite{AE}, it can be widely used in various important applications, such as digital refocusing \cite{Ng1}, the depth reconstruction \cite{LJ}, the 3-D scene reconstruction \cite{KC}, and intelligent detection \cite{LN}. Hence, it is a core problem to get a high-quality light field. 

In addition to obtaining the light field from some physical systems directly \cite{PC,WB1,WB2}, an economical, convenient, and flexible way is to reconstruct the light field from the focal stack \cite{AR} which is a set of films focused at different depths. The focal stack is defined as an integral as follows,
\begin{align}
g_{F^{\prime}}(\bar{x}_1,\bar{x}_2) 
&= |F^{\prime}|^{-2} \notag\\
&\ \times \int_{\R^2} f_{F}\left(\frac{F}{F^{\prime}} \bar{x}_1 + \left(1-\frac{F}{F^{\prime}}\right)u_1,\frac{F}{F^{\prime}} \bar{x}_2 + \left(1-\frac{F}{F^{\prime}}\right)u_2,u_1,u_2  \right) \dif u_1 u_2, \label{1.1} \end{align}
where $F$ is the depth of the reference plane from the lens plane, $F^{\prime}$ is the separation between the lens plane and the arbitrary focal plane, $g_{F^{\prime}}$ represents the focal stack at depth $F^{\prime}$, $f_{F}$ is the light field parameterized by the lens plane and the reference plane (see Figure \ref{figure1}). Equation (\ref{1.1}) is the forward operator that models the focal stack obtained by a given light field. On the contrary, if $g_{F^{\prime}}$ is given for any $F^{\prime}$, then the problem of how to reconstruct the light field from the focal stack can be mathematically expressed as an inverse problem how to get $f_F$ from $g_{F^{\prime}}$ by solving the integral equation (\ref{1.1}). 

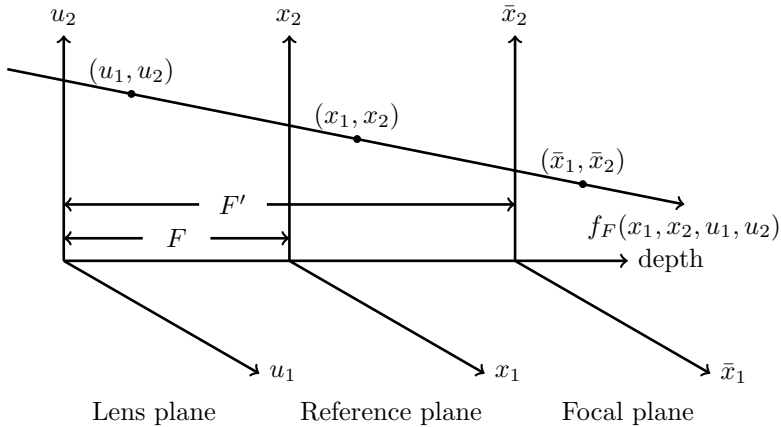
\begin{figure}[htb]
\centering
\begin{tikzpicture}[scale=1.5]
\draw[->][line width =1pt]  (-2,0) -- (-2,2) node[above] {$u_2$};
\draw[->][line width =1pt]  (-2,0) -- (-0.267,-1) node[right]{$u_1$};
\draw[->][line width =1pt]  (0,0) -- (0,2) node[above] {$x_2$};
\draw[->][line width =1pt]  (0,0) -- (1.73,-1) node[right] {$x_1$};
\draw[->][line width =1pt]  (2,0) -- (2,2) node[above] {$\bar{x}_2$};
\draw[->][line width =1pt]  (2,0) -- (3.73,-1) node[right] {$\bar{x}_1$};
\draw[->][line width =1pt]  (-2,0) -- (3,0) node[right] {depth};
\draw[->][line width =1pt]  (-2.5,1.7) -- (3.5,0.5);
\draw[->][line width =1pt]  (-0.8,0.5) -- (-2,0.5); 
\draw[->][line width =1pt]  (-0.2,0.5) -- (2,0.5); 
\draw[->][line width =1pt]  (-1.3,0.2) -- (-2,0.2); 
\draw[->][line width =1pt]  (-0.7,0.2) -- (0,0.2); 
\node at(-0.5,0.5) {$F^{\prime}$};
\node at(-1,0.2) {$F$};
\node at(3.5,0.3) {$f_F(x_1,x_2,u_1,u_2)$};
\coordinate[label=90:${(u_1,u_2)}$] (A) at (-1.4,1.48);
\coordinate[label=90:${(x_1,x_2)}$] (B) at (0.6,1.08);
\coordinate[label=90:${(\bar{x}_1,\bar{x}_2)}$] (C) at (2.6,0.68);
\filldraw (A) circle (.03)
(B) circle (.03)
(C) circle (.03);
\node at(-1.2,-1.35) {Lens plane};
\node at(3,-1.35) {Focal plane};
\node at(0.9,-1.35) {Reference plane};
\end{tikzpicture}
\caption{The light field $f_F$ in lens plane $u_1 u_2$, reference plane $x_1 x_2$, and focal plane $\bar{x}_1\bar{x}_2$. $f_F(u_1,u_2,x_1,x_2)$ is the radiance along the given ray passing through the points $(u_1,u_2)$ and $(x_1,x_2)$.}
\label{figure1}
\end{figure}

We briefly summarize previous research on equation (\ref{1.1}) from the perspective of the forward problem and the inverse problem, respectively. On the one hand, it is noticeable that the forward operator is a linear integral operator. In Ren's pioneering work \cite{Ng2}, based on a mathematical description of the forward problem, two key properties of the forward operator were given, including the Fourier slice theorem and the convolution theorem. These properties are quite similar to those satisfied by the classic Radon transform. Furthermore, other works also noticed the potential relation between the forward operator and the classic Radon transform formally \cite{LC,PF}. In terms of numerical implementation, the fast discretization of equation (\ref{1.1}) was considered in \cite{MH,NF}.

On the other hand, the inversion of the integral equation (\ref{1.1}) is an ill-posed linear inverse problem. Analogous to the analytic reconstruction method \cite{Na,PX} and algebraic iterative scheme \cite{An,JM,Qu} for the inversion of the classic Radon transform, some methods have been established in light field reconstruction from the focal stack, including the filtered back-projection (FBP)-like method \cite{LA,LC} and the weighted Landweber iterative scheme \cite{LC,YX}. Since the collection of the focal stack is under-sampled in practice, the optimization method based on the light field prior \cite{GS} and data-driven unrolling algorithm \cite{CY} were also applied to solving the integral equation (\ref{1.1}). 

Even though progress for equation (\ref{1.1}) has been made as above, four problems remain in theory. First of all, the absence of a rigorous definition for the forward operator in a more general space hinders our in-depth exploration of this light field reconstruction problem. Secondly, the relation of the Radon transform to the forward operator is not clear. Next, the properties of the forward operator are not fully obtained since the adjoint of the forward operator has not been defined so far. Finally, there are only some specific inversion methods, lacking a general analytic inversion formula.

The goal of the present paper is to study the theory of some general photography transforms. The theory provides solutions to the above four problems correspondingly. Firstly, three generalized forms of photography transform with different integral geometry characteristics are proposed, which all arise from the forward problem of the light field reconstruction from the focal stack. Secondly, the equivalence between the photography transforms and the coupled Radon transform proposed in this paper is given. Thirdly, a series of new properties satisfied by both these photography transforms and their dual operators are obtained. Fourthly, general inversion formulas for the photography transforms are established and these formulas can derive the FBP method and back-projection filtered (BPF) method.

The paper is organized as follows. In section \ref{se:Definition}, we generalize the integral equation (\ref{1.1}) from the actual scene directly, and then propose a definition of the photography transform. In section \ref{se:property}, we study the properties of the photography transform by which the previous results for the forward operator of the integral equation (\ref{1.1}) are extended. In section \ref{se:inversion}, we derive the analytic inversion formulas for the 1-dimensional and 2-dimensional photography transform as well as the corresponding analytic reconstruction algorithms. Other two types of photography transforms are defined and discussed in section \ref{se:extension}.

\section{\bf Definition of the photography transform}\label{se:Definition}

In this section, we introduce a photography transform to model the general forward problem which comes from the light field reconstruction from the focal stack. 

We begin with the physical background of a conventional camera model. The simple camera model is as follows. Let $f_F(x_1,x_2,u_1,u_2)$ be the radiance along a ray inside the camera from $(u_1,u_2)$ on the lens plane to $(x_1,x_2)$ on the sensor plane, where we use two-plane to parameterize the light field $f_F$ (see Figure \ref{figure2}) and $F$ is the distance between the plane of the lens and the plane of the sensor. Assume that the extent of the lens and sensor is infinite. The film that images on the sensor inside the conventional camera is formulated as an integral of the radiance coming through the lens under simplifications:
\begin{equation}
g_{F}(x_1,x_2) = \frac{1}{F^2} \int_{\mathbb{R}^2} f_F(x_1,x_2,u_1,u_2) \dif u_1 u_2,
\label{2.1}
\end{equation}
where $g_{F}$ denotes the film at a distance $F$ from the lens.

\begin{figure}[htb]
\centering
\begin{tikzpicture}[scale=1.5]
\draw[->][line width =1pt]  (-2,-0) -- (-2,2) node[above] {$u_2$};
\draw[->][line width =1pt]  (-2,-0) -- (-0.267,-1) node[right]{$u_1$};
\draw[->][line width =1pt]  (1,0) -- (1,2)node[above] {$x_2$};
\draw[->][line width =1pt]  (1,0) -- (2.73,-1)node[right] {$x_1$};
\draw[->][line width =1pt]  (-2.5,1.7) -- (3,0.6);
\draw[->][line width =1pt]  (-0.8,0.3) -- (-2,0.3); 
\draw[->][line width =1pt]  (-0.2,0.3) -- (1,0.3); 
\node at(-0.5,0.3) {$F$};
\node at(3.1,0.4) {$f_F(x_1,x_2,u_1,u_2)$};
\coordinate[label=90:${(u_1,u_2)}$] (A) at (-1.4,1.48);
\coordinate[label=90:${(x_1,x_2)}$] (B) at (1.6,0.88);
\filldraw (A) circle (.03)
(B) circle (.03);
\node at(-1.5,-1.35) {Lens plane};
\node at(1.5,-1.35) {Sensor plane};
\end{tikzpicture}
\caption{Two-plane parameterization of the light field $f_F(x_1,x_2,u_1,u_2)$.}
\label{figure2}
\end{figure}
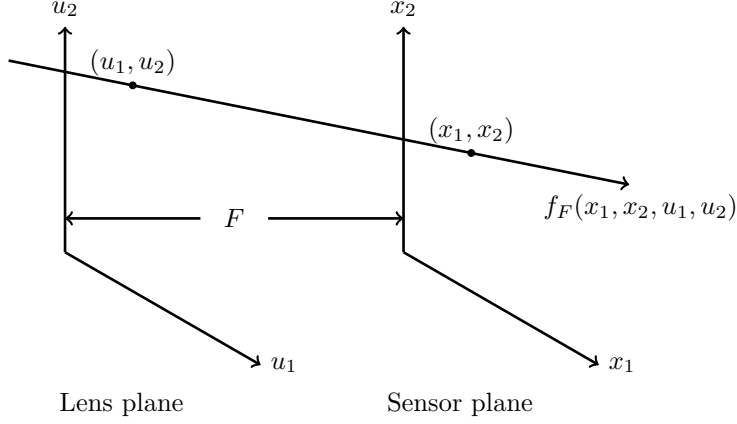

Now, we consider how to get the focal stack $g_{F^{\prime}}$ at another distance $F^{\prime}$ on the focal plane (see Figure \ref{figure1}) instead of integrating the light field  $f_{F^{\prime}}$ again like in (\ref{2.1}) directly (in this time, the sensor plane at depth $F$ is called the reference plane). In fact, we can reparameterize $f_{F^{\prime}}$ by $f_{F}$ (2-dimensional case see Figure \ref{figure3}), i.e.,
$$
f_{F^{\prime}}(\bar{x}_1,\bar{x}_2,u_1,u_2) = f_{F}\left(\frac{F}{F^{\prime}} \bar{x}_1 + \left(1-\frac{F}{F^{\prime}}\right)u_1,\frac{F}{F^{\prime}} \bar{x}_2 + \left(1-\frac{F}{F^{\prime}}\right)u_2,u_1,u_2 \right).
$$
Combining with the (\ref{2.1}), we get a useful approach (\ref{1.1}) to calculate the focal stack $g_{F^{\prime}}$ only by $f_{F}$.

\begin{figure}[htb]
\centering
\begin{tikzpicture}[scale=4]
\draw[->][line width =1pt]  (-1,-0.3) -- (-1,1) ;
\draw[->][line width =1pt]  (-0.4,-0.3) -- (-0.4,1);
\draw[->][line width =1pt]  (1,-0.3) -- (1,1);
\draw[->][line width =1pt] (-0.8,-0.1) -- (-1,-0.1);
\draw[->][line width =1pt] (-0.6,-0.1) -- (-0.4,-0.1);
\draw[->][line width =1pt] (-0.1,-0.2) -- (-1,-0.2);
\draw[->][line width =1pt] (0.1,-0.2) -- (1,-0.2);
\draw[line width =1pt] (-1.1,-0.04)--(1.1,0.84);
\draw[densely dashed][line width =1pt] (-1,0)--(1.16,0);
\node at(-0.7,-0.1) {$F$};
\node at(0,-0.2) {$F^{\prime}$};
\coordinate[label=160:$u_1$] (A) at (-1,0);
\coordinate[label=100:$x_1$] (B) at (-0.4,0.24);
\coordinate[label=100:$\bar{x}_1$] (C) at (1,0.8);
\filldraw (A) circle (.02)
(B) circle (.02)
(C) circle (.02);
\draw[densely dashed][line width =1pt] (B) -- (-0.25,0.24);
\draw[densely dashed][line width =1pt] (C) -- (1.15,0.8);
\draw [<->][line width =1pt] (-0.31,0)--(-0.31,0.24);
\draw [->][line width =1pt] (1.1,0.3)--(1.1,0);
\draw [->][line width =1pt] (1.1,0.5)--(1.1,0.8);
\node at(1.2,0.4) {$\bar{x}_1-u_1$};
\node at(-0.13,0.12) {$x_1-u_1$};
\node at(0.4,0.8) {$\frac{\bar{x}_1-u_1}{x_1-u_1}=\frac{{F}^{\prime}}{F}$};
\node at(-1,-0.35) {Lens plane};
\node at(-0.3,-0.35) {Reference plane};
\node at(1,-0.35) {Focal plane};
\end{tikzpicture}
\caption{Reparameterization of $f_{F^{\prime}}(\bar{x}_1,u_1)$ by $f_{F}(x_1,u_1)$. The figure shows the simplified 2-dimensional case involving only $u_1$, $x_1$, and $\bar{x}_1$. By similar triangles, we have $x_1 = 
\frac{F}{F^{\prime}} \bar{x}_1 + \left(1-\frac{F}{F^{\prime}}\right)u_1$. Since the ray from $u_1$ to $x_1$ and the ray from $u_1$ to $\bar{x}_1$ are the same, there holds $f_{F^{\prime}}(\bar{x}_1,u_1)=f_F\left(\frac{F}{F^{\prime}} \bar{x}_1 + \left(1-\frac{F}{F^{\prime}}\right)u_1,u_1\right)$.}
\label{figure3}
\end{figure}
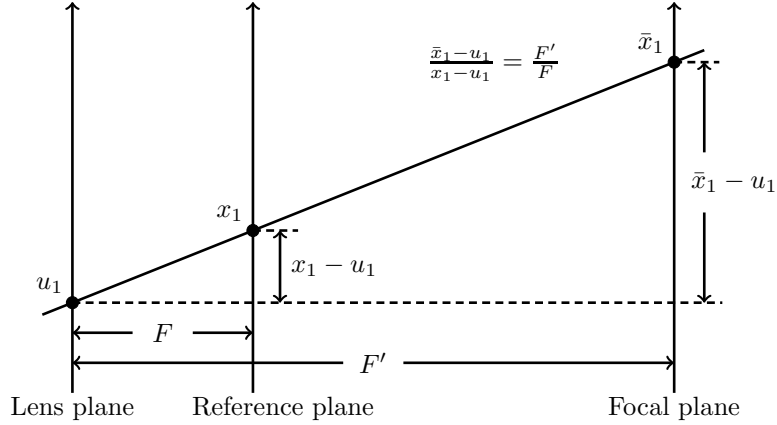

Let $\alpha = F^{\prime}/F$ for simplicity, then $F^{\prime}=\alpha F$. Without loss of generality, let $F=1$ and define $f\triangleq f_{1}$. Inspired by the above background, we put forward the  $n$-dimensional photography transform as follows.

\begin{defi}
Let $f(\boldsymbol{x},\boldsymbol{u})\in\mathcal{S}$$(\mathbb{R}^{2n})$, the Schwartz space. For $\alpha \in (\mathbb{R}-\{0\})$ and $\bar{\boldsymbol{x}} = (\bar{x}_1,\cdots,\bar{x}_n)^{\top}\in \mathbb{R}^n$, we define the photography transform $P$ of a function $f$ as
\begin{equation}
Pf(\alpha,\bar{\boldsymbol{x}}) = |\alpha|^{-n} \int_{\mathbb{R}^n} f \left(\frac{1}{\alpha}\bar{\boldsymbol{x}} + \left(1-\frac{1}{\alpha} \right)\boldsymbol{u},\boldsymbol{u}\right) \dif \boldsymbol{u},\label{2.2}
\end{equation}
where $\boldsymbol{x}= (x_1,\cdots,x_n)^{\top}\in \mathbb{R}^n$ and $\boldsymbol{u}= (u_1,\cdots,u_n)^{\top}\in \mathbb{R}^n$.
\end{defi}

\begin{remark} \leavevmode 
\rm{
(i) Obviously, $Pf$ is a function defined on $(\mathbb{R}-\{0\})\times \mathbb{R}^n$ of $\mathbb{R}^{n+1}$ and the photography transform $P$ is a linear operator which maps a function on $\mathbb{R}^{2n}$ into a function on $\mathbb{R}^{n+1}$. We also write 
\begin{equation}
P_{\alpha}f(\bar{\boldsymbol{x}}) = Pf(\alpha,\bar{\boldsymbol{x}}) \label{2.3}.
\end{equation}
If $f\in \mathcal{S}$$(\mathbb{R}^{2n})$, then $Pf$ and $P_{\alpha}f$ are in the Schwartz space on $(\mathbb{R}-\{0\})\times\mathbb{R}^{n}$ and $\mathbb{R}^{n}$ respectively. In particular, $P_{\alpha}f$ is called the focal stack for $n=2$.

(ii) To get $Pf$ from any known $f$ is a linear forward problem. Meanwhile, we also care about the ill-posed inverse problem, i.e., how to reconstruct the function $f$ from given measurements $Pf$. When $n=2$, the forward problem and the inverse problem are called refocusing \cite{Ng2} and the light field reconstruction from the focal stack \cite{LC,PF}, respectively. 

(iii) In practice, we mainly focus on cases when $n=1$ and $n=2$. Explicitly, for $f(x_1,u_1)\in \mathcal{S}$$(\mathbb{R}^2)$ $Pf$ is defined as 
\begin{equation}
Pf(\alpha,\bar{x}_1) = |\alpha|^{-1} \int_{\mathbb{R}} f\left(\frac{1}{\alpha}\bar{x}_1+ \left(1-\frac{1}{\alpha}\right)u_1,u_1\right) \dif u_1 \label{2.4}
\end{equation}
when $n=1$ and for $f(x_1,x_2,u_1,u_2)\in \mathcal{S}$$(\mathbb{R}^4)$ $Pf$ is defined as
\begin{align}
Pf(\alpha,\bar{x}_1,\bar{x}_2) 
&= |\alpha|^{-2}  \notag \\
&\ \times \int_{\mathbb{R}^2} f\left(\frac{1}{\alpha}\bar{x}_1+ \left(1-\frac{1}{\alpha}\right)u_1,\frac{1}{\alpha}\bar{x}_2 + \left(1-\frac{1}{\alpha}\right)u_2,u_1,u_2\right) \dif u_1 u_2 \label{2.5}
\end{align}
when $n=2$. Note that (\ref{1.1}) and (\ref{2.5}) coincide except for the notation of variables. Thus, the photography transform $P$ can be seen as a high-dimensional extension for the forward operator behind the forward problem of the light field reconstruction problem.   
}
\end{remark}

\section{\bf Properties of the photography transform}\label{se:property}

In this section, we mainly study the properties of both the photography transform defined in the above section and its dual operator.

\subsection{Relation of the photography transform with the Radon transform}

To clarify the relationship between the photography transform and the Radon transform, we need to define a new special transform, that is, the coupled Radon transform.

\begin{defi}
Let $f(\boldsymbol{x},\boldsymbol{u}) \in \mathcal{S}$$(\mathbb{R}^{2n})$. For $\alpha \in \mathbb{R}$ and $\bar{\boldsymbol{x}} \in \mathbb{R}^{n}$, the coupled Radon transform of $f$ is defined by 
\begin{equation}
R_c f((\alpha,1-\alpha),\bar{\boldsymbol{x}}) = \int_{\mathbb{R}^{n}}\int_{\mathbb{R}^{n}} f(\boldsymbol{x},\boldsymbol{u}) \delta (\alpha \boldsymbol{x}+(1-\alpha) \boldsymbol{u} - \bar{\boldsymbol{x}})\dif \boldsymbol{x} \dif \boldsymbol{u}, \label{3.1}
\end{equation}
where $\delta$ represents the $\delta$-function. 
\end{defi}

\begin{remark}\label{re2}{\rm
Since the multiple dimensional delta function $\delta(\boldsymbol{x})=\delta(x_{1})\times\cdots\times\delta(x_{n})$ corresponds to $\prod_{i=1}^{n}\delta(x_{i})$ \cite{GM}, the $n$-dimensional $\delta$-function in the integral (\ref{3.1}) can be expressed as the product of $n$ 1-dimensional $\delta$-functions.
}
\end{remark}

We recall two types of the classic Radon transforms $R$ and $R_1$ \cite{Ra} defined by
$$
Rf(\boldsymbol{\theta},p) = \int_{\boldsymbol{\theta}\cdot (\boldsymbol{x}, \boldsymbol{u})=p}f(\boldsymbol{x},\boldsymbol{u}) \dif s, 
$$
for $\boldsymbol{\theta}\in S^{2n-1}$, the unite sphere of $\mathbb{R}^{2n}$ and $p\in \mathbb{R}$, and
$$
R_1 f(\boldsymbol{\alpha},p) = \int_{\mathbb{R}^{n}}\int_{\mathbb{R}^{n}} f(\boldsymbol{x},\boldsymbol{u})\delta(\boldsymbol{\alpha}\cdot (\boldsymbol{x},\boldsymbol{u}) - p)\dif \boldsymbol{x} \dif \boldsymbol{u}
$$
for $\boldsymbol{\alpha}\in \mathbb{R}^{2n}$, $p\in \mathbb{R}$. It is noted that there are two differences between the definitions of the coupled Radon transform $R_c$ and classic Radon transform $R_1$. Firstly, the integrand in the integral of $R_1 f$ contains only one 1-dimensional $\delta$-function, while that of $R_cf$ contains the product of $n$ 1-dimensional $\delta$-functions (see Remark \ref{re2}). Moreover, the variables $(\boldsymbol{x},\boldsymbol{u})$ in the integral of $R_1 f$ are multiplied by the distinct elements in $\boldsymbol{\alpha}$ independently. But for any pair of variables $(x_i,u_i)$ from $(\boldsymbol{x},\boldsymbol{u})$ in the integral of $R_c f$, $x_1,\ldots,x_n$ are coupled by multiplying the same $\alpha$, and $u_1,\ldots,u_n$ are multiplied by the same $(1-\alpha)$. Therefore $R_c$ is called the coupled Radon transform due to the coupling relation between variables. 

Then we can establish the equivalent relation between the photography transform $P$ and the coupled Radon transform $R_c$ by the following theorem.

\begin{theorem}
If $f(\boldsymbol{x},\boldsymbol{u}) \in \mathcal{S}$$(\mathbb{R}^{2n})$, then for any $\alpha \in (\mathbb{R}-\{0\})$ and $\bar{\boldsymbol{x}}\in \mathbb{R}^n$
\begin{equation}
Pf(\alpha,\bar{\boldsymbol{x}}) = R_c f((\alpha,1-\alpha),\bar{\boldsymbol{x}}). \label{th3.2}
\end{equation}    
\end{theorem}

\begin{proof}
By variables substitutions $\boldsymbol{x}=\alpha \boldsymbol{x}^{\prime}+(1-\alpha) \boldsymbol{u}^{\prime}$ and $\boldsymbol{u}= \boldsymbol{u}^{\prime}$, for any $\alpha \in (\mathbb{R}-\{0\})$ we have
\begin{equation*}
\begin{split}
    &\int_{\mathbb{R}^{n}}\int_{\mathbb{R}^{n}} f(\boldsymbol{x}^{\prime},\boldsymbol{u}^{\prime}) \delta (\alpha \boldsymbol{x}^{\prime}+(1-\alpha) \boldsymbol{u}^{\prime} - \bar{\boldsymbol{x}})\dif \boldsymbol{x}^{\prime} \dif \boldsymbol{u}^{\prime}  \\
={}& |\alpha|^{-n}\int_{\mathbb{R}^{n}}\int_{\mathbb{R}^{n}} f\left(\frac{1}{\alpha}\boldsymbol{x}+\left(1-\frac{1}{\alpha} \right) \boldsymbol{u},\boldsymbol{u}\right) \delta (\boldsymbol{x}- \bar{\boldsymbol{x}} )\dif \boldsymbol{x} \dif \boldsymbol{u}  \\
={}& |\alpha|^{-n} \int_{\mathbb{R}^n} f \left(\frac{1}{\alpha}\bar{\boldsymbol{x}} + \left(1-\frac{1}{\alpha} \right)\boldsymbol{u},\boldsymbol{u}\right) \dif \boldsymbol{u},
\end{split}
\end{equation*}
which completes the proof.
\end{proof}

By the coupled Radon transform, the range of the photography transform $P$ can be extended to $\mathcal{S}$$(\mathbb{R}\times \mathbb{R}^{n})$ and then in the following we assume that $Pf$ is defined on $\mathbb{R}\times \mathbb{R}^{n}$.

For $n=1$, there is a stronger equivalent relation between the photography transform and the classic Radon transform.

\begin{cor}\label{cor3.3}
For $n=1$ we have 
\begin{align*}
Pf(\alpha,\bar{x}_1) 
&=  R_1 f((\alpha,1-\alpha),\bar{x}_1)\\
&= \frac{1}{\sqrt{\alpha^2+(1-\alpha)^2}} Rf \left( \frac{(\alpha,1-\alpha)}{\| (\alpha,1-\alpha)\|}, \frac{\bar{x}_1}{\| (\alpha,1-\alpha)\|} \right).    
\end{align*}
\end{cor}

\begin{proof}
The first equality is a direct corollary of (\ref{th3.2}) by letting $n=1$, and the second equality is guaranteed by the fact that $R_1 f$ is a homogeneous function of degree -1 \cite{Ra}.
\end{proof}

\subsection{Fourier slice theorem and convolution theorem}
Similar to the classic Radon transform, many important properties of the photography transform are related to the Fourier transform and convolution. In this paper, whenever the Fourier transform or convolution of functions on $\mathbb{R}\times \mathbb{R}^{n}$ is applied it is to be taken with respect to the second variable, i.e., for any $h,g \in \mathcal{S}$$(\mathbb{R}\times \mathbb{R}^{n})$, 
\begin{gather}
h*g(\alpha,\bar{\boldsymbol{x}}) = \int_{\mathbb{R}^n} h(\alpha,\bar{\boldsymbol{x}}-\bar{\boldsymbol{x}}^{\prime}) g(\alpha,\bar{\boldsymbol{x}}^{\prime}) \dif \bar{\boldsymbol{x}}^{\prime},  \label{convolution}\\
\hat{h}(\alpha,\xi_{\bar{\boldsymbol{x}}}) = \int_{\mathbb{R}^n} e^{-2 \pi i \bar{\boldsymbol{x}} \cdot  \xi_{\bar{\boldsymbol{x}}}} h(\alpha,\bar{\boldsymbol{x}}) \dif \bar{\boldsymbol{x}}, \label{Fourier}
\end{gather}
where $\xi_{\bar{\boldsymbol{x}}}=(\xi_{\bar{x}_1},\ldots,\xi_{\bar{x}_n})^{\top}\in \mathbb{R}^{n}$.

Here we have two properties of the photography transform $P$ related to the Fourier transform and convolution.

\begin{theorem}\label{FandC}
For $f(\boldsymbol{x},\boldsymbol{u}),\ g(\boldsymbol{x},\boldsymbol{u}) \in \mathcal{S}$$(\mathbb{R}^{2n})$, we have
\newline{\rm (}i{\rm )}
\begin{equation}
(P_{\alpha}f) \, \hat{}\, (\xi_{\bar{\boldsymbol{x}}}) = \hat{f}(\alpha \xi_{\bar{\boldsymbol{x}}},(1-\alpha) \xi_{\bar{\boldsymbol{x}}}), \quad \alpha \in \mathbb{R}. \label{thm3.4.1}
\end{equation}
{\rm (}ii{\rm )}
\begin{equation}
P_{\alpha}(f*g) = P_{\alpha}f * P_{\alpha}g, \quad \alpha \in \mathbb{R}. \label{thm3.4.2}
\end{equation}
\end{theorem} 

\begin{proof}
Using the Fourier transform defined by (\ref{Fourier}), we get
\begin{equation*}
\begin{split}
(P_{\alpha}f) \,\hat{}\,  (\xi_{\bar{\boldsymbol{x}}}) 
&= \int_{\mathbb{R}^n} e^{-2 \pi i \bar{\boldsymbol{x}} \cdot  \xi_{\bar{\boldsymbol{x}}}} P_{\alpha}f(\alpha,\bar{\boldsymbol{x}}) \dif \bar{\boldsymbol{x}} \\
&= |\alpha|^{-n} \int_{\mathbb{R}^n} e^{-2 \pi i \bar{\boldsymbol{x}} \cdot  \xi_{\bar{\boldsymbol{x}}}} \int_{\mathbb{R}^n} f \left( \frac{1}{\alpha} \bar{\boldsymbol{x}} + \left(1-\frac{1}{\alpha} \right) \bar{\boldsymbol{u}},\bar{\boldsymbol{u}} \right) \dif \bar{\boldsymbol{u}}  \dif  \bar{\boldsymbol{x}}. 
\end{split}
\end{equation*}
Making the substitutions $\boldsymbol{u}=\bar{\boldsymbol{u}}$ and $\boldsymbol{x}=\frac{1}{\alpha} \bar{\boldsymbol{x}} + \left(1-\frac{1}{\alpha} \right) \bar{\boldsymbol{u}}$ we obtain
\begin{equation*}
\begin{split}
(P_{\alpha}f)\,\hat{}\, (\xi_{\bar{\boldsymbol{x}}}) 
&= \int_{\mathbb{R}^n} \int_{\mathbb{R}^n} e^{-2 \pi i (\alpha \boldsymbol{x}+ (1-\alpha)\boldsymbol{u}) \cdot  \xi_{\bar{\boldsymbol{x}}}} f(\boldsymbol{x},\boldsymbol{u}) \dif \boldsymbol{x} \dif\boldsymbol{u}   \\
&= \hat{f}(\alpha \xi_{\bar{\boldsymbol{x}}},(1-\alpha) \xi_{\bar{\boldsymbol{x}}}). 
\end{split}
\end{equation*}
This proves the first formula. 

By (\ref{thm3.4.1}) and an equality $(f*g)\,\hat{}  = \hat{f}\hat{g}$,
\begin{equation*}
\begin{split}
(P_{\alpha}(f*g))\,\hat{}\, (\xi_{\bar{\boldsymbol{x}}}) &= (f*g)\,\hat{}\, (\alpha \xi_{\bar{\boldsymbol{x}}},(1-\alpha) \xi_{\bar{\boldsymbol{x}}}) \\
&= (\hat{f}\hat{g})(\alpha \xi_{\bar{\boldsymbol{x}}},(1-\alpha) \xi_{\bar{\boldsymbol{x}}}) \\
&= \hat{f}(\alpha \xi_{\bar{\boldsymbol{x}}},(1-\alpha) \xi_{\bar{\boldsymbol{x}}}) \ \hat{g}(\alpha \xi_{\bar{\boldsymbol{x}}},(1-\alpha) \xi_{\bar{\boldsymbol{x}}})  \\
&= (P_{\alpha}(f))\,\hat{}\, (\xi_{\bar{\boldsymbol{x}}}) \ (P_{\alpha}(g))\,\hat{}\,  (\xi_{\bar{\boldsymbol{x}}}).
\end{split}
\end{equation*}
The second formula follows immediately by applying inverse Fourier transform to both sides of the above equality.
\end{proof} 

\begin{remark}\leavevmode\label{R3.5}
\rm{
(i) Theorem \ref{FandC} extends the results in \cite{Ng2} just for $n=2$ to the arbitrary $n$-dimensional case and is proven more directly than the proof by a so-called generalized Fourier slice theorem used in \cite{Ng2}. According to the definition (\ref{convolution}) and (\ref{Fourier}) at the beginning of this subsection, Theorem  \ref{FandC} is also true if we substitute $P_{\alpha}$ by $P$. This result shows that the photography transform $P$ retains or partially inherits those properties of the classic Radon transform about the convolution and Fourier transform in \cite{Na}. 

(ii) Theorem \ref{FandC} (i) is also called the projection theorem or Fourier slice theorem of the photography transform $P$. This theorem states that taking the $n$-dimensional Fourier transform to $P_{\alpha}f$ is equivalent to picking out the value of $\hat{f}(\xi_{\boldsymbol{x}},\xi_{\boldsymbol{u}})$ on the 1-dimensional collection of $n$-dimensional subspaces
\begin{equation}
\bigcup_{\alpha \in \mathbb{R}}\left\{ (\xi_{\boldsymbol{x}},\xi_{\boldsymbol{u}})\in \mathbb{R}^{2n}:  \xi_{\boldsymbol{x}}= \alpha\xi_{\bar{\boldsymbol{x}}},\ \xi_{\boldsymbol{u}}= (1-\alpha) \xi_{\bar{\boldsymbol{x}}},\  \xi_{\bar{\boldsymbol{x}}}\in \mathbb{R}^{n}\right\} \label{remark3.6}.
\end{equation}
For $n=1$, it is noted that the whole region in the frequency domain of $\hat{f}$ can be covered by taking the Fourier transform to $Pf$. But in general, we can not do the same thing when $n\geq 2$ since $(P_{\alpha}f) \, \hat{}$ is only known on the 1-dimensional collection of $n$-dimensional subspaces (\ref{remark3.6}) instead of the entire frequency domain. (\ref{thm3.4.1}) can be used as a basis of the Fourier reconstruction method, however, one will face the difficulty of serious insufficient information in the frequency domain for $n\geq 2$.

(iii) Theorem \ref{FandC} (ii) is the convolution theorem of the photography transform $P$. This theorem can be intuitively understood as a fact that if the real light field is given by convolving an ideal light field $f$ with a function $g$, then the film $P_{\alpha}(f*g)$ obtained from the real light field $f*g$ is equivalent to the convolution of the film $P_{\alpha}f$ obtained from the ideal light field $f$ with the function $P_{\alpha}g$. Assuming that $g$ is known, we can either simulate the film in reality from the given ideal film or recover the ideal film from the given real film.
}
\end{remark}

\subsection{dual operator}
In the last subsection, we will discuss the dual operator of the photography transform and its related properties. To define dual operators $P_{\alpha}^{*}$, $P^{*}$, we begin with 
\begin{align}
\int_{\mathbb{R}^n} P_{\alpha}f(\bar{\boldsymbol{x}})g(\bar{\boldsymbol{x}})\dif \bar{\boldsymbol{x}} 
&=|\alpha|^{-n} \int_{\mathbb{R}^n}\int_{\mathbb{R}^n}f \left( \frac{1}{\alpha} \bar{\boldsymbol{x}} + \left(1-\frac{1}{\alpha} \right) \bar{\boldsymbol{u}},\bar{\boldsymbol{u}} \right)  g(\bar{\boldsymbol{x}}) \dif \bar{\boldsymbol{u}} \dif \bar{\boldsymbol{x}} \nonumber \\
&= \int_{\mathbb{R}^n}\int_{\mathbb{R}^n} f(\boldsymbol{x},\boldsymbol{u})g(\alpha \boldsymbol{x}+(1-\alpha)\boldsymbol{u}) \dif \boldsymbol{x} \dif \boldsymbol{u}. \label{3.3.1}
\end{align}
Defining 
$$
(P^{*}_{\alpha}g)(\boldsymbol{x},\boldsymbol{u}) = g(\alpha \boldsymbol{x}+(1-\alpha)\boldsymbol{u}).
$$
and combining with (\ref{3.3.1}) we have 
$$
\int_{\mathbb{R}^n} P_{\alpha}f(\bar{\boldsymbol{x}})g(\bar{\boldsymbol{x}})\dif \bar{\boldsymbol{x}} = \int_{\mathbb{R}^n}\int_{\mathbb{R}^n} f(\boldsymbol{x},\boldsymbol{u}) (P^{*}_{\alpha}g)(\boldsymbol{x},\boldsymbol{u}) \dif \boldsymbol{x} \dif \boldsymbol{u}.
$$
Integrating over $\mathbb{R}$ we get

$$
\int_{\mathbb{R}}\int_{\mathbb{R}^n} Pf(\alpha, \bar{\boldsymbol{x}})g(\alpha, \bar{\boldsymbol{x}})\dif \bar{\boldsymbol{x}} \dif \alpha =  \int_{\mathbb{R}^n}\int_{\mathbb{R}^n} f(\boldsymbol{x},\boldsymbol{u}) P^{*}g(\boldsymbol{x},\boldsymbol{u}) \dif \boldsymbol{x} \dif \boldsymbol{u}.
$$
Thus
\begin{equation}
P^{*}g(\boldsymbol{x},\boldsymbol{u}) = \int_{\mathbb{R}} g(\alpha,\alpha \boldsymbol{x}+(1-\alpha)\boldsymbol{u})\dif \alpha. \label{3.3.3}
\end{equation}
Note that $P$ and $P^{*}$ form a dual pair in the sense of integral geometry: while $P$ integrates overall points in the linear manifold 
$$
\mathcal{M}=\{(\boldsymbol{x},\boldsymbol{u})\in\mathbb{R}^{2n}:\alpha \boldsymbol{x}+(1-\alpha) \boldsymbol{u} - \bar{\boldsymbol{x}}=0\}, 
$$
$P^{*}$ integrates over all linear manifolds 
$$
\mathcal{M}_{\alpha}=\left\{(\boldsymbol{x}^{\prime},\boldsymbol{u}^{\prime})\in\mathbb{R}^{2n}:\alpha \boldsymbol{x}^{\prime}+(1-\alpha) \boldsymbol{u}^{\prime}=\alpha \boldsymbol{x}+(1-\alpha) \boldsymbol{u} \right\}
$$
with respect to $\alpha$ through the point $(\boldsymbol{x},\boldsymbol{u})$. 

The following property combines the dual operator $P^{*}$ with the convolution.

\begin{theorem}\label{dual}
For $f(\boldsymbol{x},\boldsymbol{u}) \in \mathcal{S}$$(\mathbb{R}^{2n})$ and $g(\alpha,\bar{\boldsymbol{x}}) \in \mathcal{S}$$(\mathbb{R}^{n+1})$, we have
\begin{equation}
(P^{*}g)*f = P^{*}(g*Pf). \label{thm3.6}
\end{equation}
\end{theorem}

\begin{proof}
By (\ref{3.3.3}), we have 
\begin{align}
& ((P^{*}g)*f)(\boldsymbol{x},\boldsymbol{u}) \notag\\
{}=& \int_{\mathbb{R}^n}\int_{\mathbb{R}^n} P^{*}g(\boldsymbol{x}-\boldsymbol{x}^{\prime},\boldsymbol{u}-\boldsymbol{u}^{\prime}) f(\boldsymbol{x}^{\prime},\boldsymbol{u}^{\prime}) \dif \boldsymbol{x}^{\prime}\dif \boldsymbol{u}^{\prime} \nonumber \\
{}=&\int_{\mathbb{R}^n}\int_{\mathbb{R}^n}\int_{\mathbb{R}} g(\alpha, \alpha(\boldsymbol{x}-\boldsymbol{x}^{\prime})+(1-\alpha)(\boldsymbol{u}-\boldsymbol{u}^{\prime}))f(\boldsymbol{x}^{\prime},\boldsymbol{u}^{\prime}) \dif \alpha \dif \boldsymbol{x}^{\prime}\dif \boldsymbol{u}^{\prime} \label{3.3.4}.
\end{align}
which is the left side of (\ref{thm3.6}). Using (\ref{convolution}), (\ref{3.3.3}) and (\ref{2.2}) for the right side of (\ref{thm3.6}) we obtain
\begin{align*}
&(P^{*}(g*Pf))(\boldsymbol{x},\boldsymbol{u})\\ 
&= \int_{\mathbb{R}} (g*Pf)(\alpha,\alpha \boldsymbol{x}+(1-\alpha)\boldsymbol{u}) \dif \alpha\\
&= \int_{\mathbb{R}} \int_{\mathbb{R}^n} g(\alpha,\alpha \boldsymbol{x}+(1-\alpha)\boldsymbol{u}-\bar{\boldsymbol{x}}) Pf(\alpha,\bar{\boldsymbol{x}}) \dif\bar{\boldsymbol{x}} \dif \alpha \\
&= |\alpha|^{-n} \int_{\mathbb{R}} \int_{\mathbb{R}^n}\int_{\mathbb{R}^n}g(\alpha,\alpha \boldsymbol{x}+(1-\alpha)\boldsymbol{u}-\bar{\boldsymbol{x}}) f \left(\frac{1}{\alpha}\bar{\boldsymbol{x}} + \left(1-\frac{1}{\alpha} \right)\bar{\boldsymbol{u}},\bar{\boldsymbol{u}}\right) \dif \bar{\boldsymbol{u}} \dif \bar{\boldsymbol{x}} \dif \alpha.
\end{align*}
Changing variables by letting $\boldsymbol{u}^{\prime}=\bar{\boldsymbol{u}}$ and $\boldsymbol{x}^{\prime}=\frac{1}{\alpha} \bar{\boldsymbol{x}} + \left(1-\frac{1}{\alpha} \right) \bar{\boldsymbol{u}}$, we get
\begin{equation}
(P^{*}(g*Pf))(\boldsymbol{x},\boldsymbol{u}) = \int_{\mathbb{R}} \int_{\mathbb{R}^n}\int_{\mathbb{R}^n}g(\alpha, \alpha(\boldsymbol{x}-\boldsymbol{x}^{\prime})+(1-\alpha)(\boldsymbol{u}-\boldsymbol{u}^{\prime}))f(\boldsymbol{x}^{\prime},\boldsymbol{u}^{\prime})  \dif \boldsymbol{x}^{\prime}\dif \boldsymbol{u}^{\prime} \dif \alpha \label{3.3.5}.
\end{equation}
Then the proof is completed by comparing (\ref{3.3.4}) with (\ref{3.3.5}).
\end{proof}

\begin{remark}
\rm{ Theorem \ref{dual} implies that $P^{*}$ preserves the convolution property of the dual operator for the classic Radon transform in \cite{Na}.
}
\end{remark}

It can be found from the above theorems that some properties of the photography transform $P$ are consistent with or close to those of the classic Radon transform. However, the following theorem shows us an intrinsic difference between two transforms. Then we conclude this section with the convolution representation of the normal operator $P^{*}P$ and related useful discussion.

\begin{theorem}\label{normal}
Let $f(\boldsymbol{x},\boldsymbol{u}) \in \mathcal{S}$$(\mathbb{R}^{2n})$. Then for any $j=1,\ldots,n$, we have

\begin{equation}
P^{*}Pf = \frac{|x_j - u_j|^{n-2} \prod_{i=1,i\neq j}^{n}\delta \big(\sin(\phi_j - \phi_i)\big)}{\| (x_j,u_j)\|^{n-1} \prod_{i=1,i\neq j}^{n}\|(x_i,u_i) \|} *f  \label{thm3.8},
\end{equation}
where for any finite sequence $a_i$, $\prod_{i=1,i\neq j}^{n} a_i$ is taken as $1$ when $n=1$, and $\phi_i={\rm arg}(x_i,u_i)$, the principal argument of $(x_i,u_i)$, i.e., $(x_i,u_i)=r(\cos(\phi_i),\sin(\phi_i))^{\top}$, $\phi_i \in [0,2\pi)$, $r\in [0,\infty)$. 
\end{theorem}

\begin{proof}
By (\ref{3.3.3}) and the equivalent definition of photography transform (\ref{3.1}), we have
\begin{align*}
& P^{*}Pf(\boldsymbol{x},\boldsymbol{u}) \\
&= \int_{\mathbb{R}} Pf(\alpha,\alpha \boldsymbol{x}+(1-\alpha)\boldsymbol{u}) \dif\alpha \\
&=\int_{\mathbb{R}}\int_{\mathbb{R}^{n}}\int_{\mathbb{R}^{n}}f(
\boldsymbol{x}^{\prime},\boldsymbol{u}^{\prime}) \delta(\alpha(\boldsymbol{x}^{\prime}-\boldsymbol{x})+(1-\alpha)(\boldsymbol{u}^{\prime}-\boldsymbol{u}))\dif \boldsymbol{x}^{\prime}\dif \boldsymbol{u}^{\prime} \dif \alpha  \\
&= \int_{\mathbb{R}^{n}}\int_{\mathbb{R}^{n}} \frac{f(\boldsymbol{x}^{\prime},\boldsymbol{u}^{\prime})}{\prod_{i=1}^{n}\| (x_i-x^{\prime}_i,u_i-u^{\prime}_i) \|} 
\int_{\mathbb{R}} [\alpha^2+(1-\alpha)^2]^{-\frac{n}{2}}\prod_{i=1}^{n} \delta \bigg(\frac{\alpha}{\sqrt{\alpha^2+(1-\alpha)^2}} \\
&\quad \times  \frac{x_i - x^{\prime}_i}{\|(x_i-x^{\prime}_i,u_i-u^{\prime}_i) \|}- \frac{\alpha - 1}{\sqrt{\alpha^2+(1-\alpha)^2}} \frac{u_i - u^{\prime}_i}{\|(x_i-x^{\prime}_i,u_i-u^{\prime}_i)\|} \bigg) \dif \alpha \dif \boldsymbol{x}^{\prime}\dif \boldsymbol{u}^{\prime}.
\end{align*}
For $i=1,\ldots,n$, denote
$$
\cos \phi_i = \frac{x_i - x^{\prime}_i}{\|(x_i-x^{\prime}_i,u_i-u^{\prime}_i)\|}, \quad
\sin \phi_i = \frac{u_i - u^{\prime}_i}{\|(x_i-x^{\prime}_i,u_i-u^{\prime}_i)\|}.
$$
Breaking the inner integral into the positive and negative parts and making the substitution $\tan \theta = 1 - \frac{1}{\alpha}$ ($ \dif \alpha= (\cos \theta - \sin \theta)^{-2}\dif \theta $), we have
\begin{align*}
&\int_{0}^{+\infty} + \int_{-\infty}^{0} [\alpha^2+(1-\alpha)^2]^{-\frac{n}{2}} \prod_{i=1}^{n} \delta \bigg(\frac{\alpha}{\sqrt{\alpha^2+(1-\alpha)^2}} \cos \phi_i \\
&\quad - \frac{\alpha - 1}{\sqrt{\alpha^2+(1-\alpha)^2}} \sin \phi_i \bigg) \dif \alpha \\
{}=& \int_{-\frac{\pi}{2}}^{\frac{\pi}{4}} + \int_{\frac{\pi}{4}}^{\frac{\pi}{2}} | \cos \theta - \sin \theta |^{n-2} \prod_{i=1}^{n} \delta (\cos\theta \cos \phi_i - \sin \theta \sin \phi_i) \dif \theta \\
{}=& \int_{-\frac{\pi}{2}}^{\frac{\pi}{2}} | \cos \theta - \sin \theta |^{n-2} \prod_{i=1}^{n} \delta\big(\cos (\theta + \phi_i) \big)\dif \theta 
\end{align*}
by the following equalities
\begin{align*}
\alpha^2+(1-\alpha)^2]^{-\frac{1}{2}} &= | \cos \theta - \sin \theta |, \\
\frac{\alpha}{\sqrt{\alpha^2+(1-\alpha)^2}} &= \frac{|\cos \theta - \sin \theta|}{\cos \theta - \sin \theta} \cos \theta, \\
\frac{\alpha - 1}{\sqrt{\alpha^2+(1-\alpha)^2}} &= \frac{|\cos \theta - \sin \theta|}{\cos \theta - \sin \theta} \sin \theta.
\end{align*}
Note that
\begin{align*}
&\int_{\frac{\pi}{2}}^{\frac{3\pi}{2}} | \cos \theta - \sin \theta |^{n-2} \prod_{i=1}^{n} \delta\big(\cos (\theta + \phi_i) \big)\dif \theta \\
{}=& \int_{-\frac{\pi}{2}}^{\frac{\pi}{2}} | \cos \theta - \sin \theta |^{n-2} \prod_{i=1}^{n} \delta\big(\cos (\theta + \phi_i)\big)\dif \theta.    
\end{align*}
By 2$\pi$-periodicity of sine and cosine functions, for any $j=1,\ldots,n$ we get
\begin{align*}
&\int_{-\frac{\pi}{2}}^{\frac{\pi}{2}} | \cos \theta - \sin \theta |^{n-2} \prod_{i=1}^{n} \delta\big(\cos (\theta + \phi_i)\big)\dif \theta  \\
={}&\frac{1}{2} \int_{-\frac{\pi}{2}}^{\frac{3\pi}{2}} | \cos \theta - \sin \theta |^{n-2} \prod_{i=1}^{n} \delta\big(\cos (\theta + \phi_i)\big)\dif \theta \\
={}& \frac{1}{2} \int_{-\frac{\pi}{2}-\phi_j}^{\frac{3\pi}{2}-\phi_j} \big| \cos (\theta -\phi_j)- \sin (\theta - \phi_j) \big|^{n-2} \delta(\cos \theta ) \prod_{i=1,i \neq j}^{n} \delta\big(\cos (\theta + \phi_i)\big) \dif \theta \\
={}& \frac{1}{2} \int_{0}^{2\pi} \big| \cos (\theta -\phi_j)- \sin (\theta - \phi_j) \big|^{n-2} \delta(\cos \theta) \prod_{i=1,i \neq j}^{n} \delta\big(\cos (\theta + \phi_i - \phi_j)\big) \dif \theta\\
={}& \int_{0}^{\pi} \big|(\cos \phi_j + \sin \phi_j)\cos\theta  +  (\sin \phi_j - \cos \phi_j)\sin \theta \big|^{n-2} \delta(\cos \theta )\\
& \qquad \qquad \qquad \qquad \times \prod_{i=1,i \neq j}^{n} \delta\big( \cos(\phi_j - \phi_i)\cos \theta + \sin(\phi_j - \phi_i) \sin \theta \big)\dif \theta.
\end{align*}
Letting $t=\cos\theta$, we obtain
\begin{align*}
& \int_{0}^{\pi} \big|(\cos \phi_j + \sin \phi_j)\cos\theta  +  (\sin \phi_j - \cos \phi_j)\sin \theta \big|^{n-2} \delta(\cos \theta ) \\
& \qquad  \times \prod_{i=1,i \neq j}^{n} \delta\big( \cos(\phi_j - \phi_i)\cos \theta + \sin(\phi_j - \phi_i) \sin \theta \big)\dif \theta \\
={}&  \int_{-1}^{1} \big|(\cos \phi_j + \sin \phi_j)t +(\sin \phi_j - \cos \phi_j)(1-t^2)^{\frac{1}{2}} \big|^{n-2} \delta(t)  \\
& \qquad  \times \prod_{i=1,i \neq j}^{n} \delta\big( \cos(\phi_j - \phi_i)t + \sin(\phi_j - \phi_i) (1-t^2)^{\frac{1}{2}} \big)(1-t^2)^{-\frac{1}{2}} \dif t  \\ 
={}& |\sin \phi_j - \cos \phi_j|^{n-2} \prod_{i=1,i \neq j}^{n} \delta\big(\sin(\phi_j - \phi_i) \big). 
\end{align*}
Combining with the definition of $\phi_j$, we have
\begin{align*}
&P^{*}Pf(\boldsymbol{x},\boldsymbol{u})\\
&= \int_{\mathbb{R}^{n}}\int_{\mathbb{R}^{n}} \frac{\left|\sin \phi_j - \cos \phi_j \right|^{n-2}}{\prod_{i=1}^{n}\| (x_i-x^{\prime}_i,u_i-u^{\prime}_i) \|} f(\boldsymbol{x}^{\prime},\boldsymbol{u}^{\prime})  \prod_{i=1,i \neq j}^{n} \delta\big(\sin(\phi_j - \phi_i) \big)   \dif \boldsymbol{x}^{\prime}\dif \boldsymbol{u}^{\prime} \\
&= \int_{\mathbb{R}^n}\int_{\mathbb{R}^n}\frac{\left|(x_j-x^{\prime}_j)-(u_j-u^{\prime}_j)\right|^{n-2}\prod_{i=1,i \neq j}^{n} \delta\big(\sin(\phi_j - \phi_i) \big)}{\left\|(x_j-x^{\prime}_j,u_j-u^{\prime}_j) \right\|^{n-1}\prod_{i=1,i \neq j}^{n}\left\|(x_i-x^{\prime}_i,u_i-u^{\prime}_i) \right\|} f(\boldsymbol{x}^{\prime},\boldsymbol{u}^{\prime}) \dif \boldsymbol{x}^{\prime}\dif \boldsymbol{u}^{\prime}, 
\end{align*}
Hence, the result follows.
\end{proof}

\begin{remark}
\rm{ By checking the proof, we can get an equivalent form of (\ref{thm3.8}) only with respect to $(\boldsymbol{x},\boldsymbol{u})$. Since
$$
\delta\left(\sin(\phi_j - \phi_i) \right)=\frac{\delta\big( (u_j-u^{\prime}_j)(x_i-x^{\prime}_i)-(u_i-u^{\prime}_i)(x_j-x^{\prime}_j)\big)}{\|(x_i-x^{\prime}_i,u_i-u^{\prime}_i) \|\|(x_j-x^{\prime}_j,u_j-u^{\prime}_j) \|},
$$
then 
$$
P^{*}Pf = \left\{ |x_j - u_j|^{n-2} \prod_{i=1,i \neq j}^{n}\delta\big( u_j x_i-u_i x_j \big) \right\} *f.
$$
}
\end{remark}

\begin{cor}
For $f(\boldsymbol{x},\boldsymbol{u})\in\mathcal{S}$$(\mathbb{R}^{2n})$.\\
{\rm (}i{\rm )} For $n=1$, the formula (\ref{thm3.8}) reads as 
\begin{equation}
P^{*}Pf = |x_1-u_1|^{-1}*f.
\end{equation}
{\rm (}ii{\rm )} For $n=2$, the formula (\ref{thm3.8}) reads as
\begin{equation}
P^{*}Pf = \frac{\delta(\sin(\phi -\varphi))}{\|(x_1,u_1)\| \|(x_2,u_2)\|}*f ,
\end{equation}
where $\phi = {\rm arg}(x_1,u_1)$ and $\varphi = {\rm arg}(x_2,u_2)$.
\end{cor}

\begin{remark}\leavevmode 
\rm{ For $f(\boldsymbol{x},\boldsymbol{u}) \in \mathcal{S}$$(\mathbb{R}^{2n})$, the normal operator of the classic Radon transform is represented by \cite{Na}
$$
R^{\#}Rf = |S^{2n-2}| \|(\boldsymbol{x},\boldsymbol{u})\|^{-1}*f,
$$
where $R^{\#}$ is the dual operator of $R$ and $|S^{m-1}|$ is the surface of the unit sphere $S^{m-1}$ in $\mathbb{R}^{m}$. Comparing this formula with the corresponding representation for the photography transform in Theorem \ref{normal}, we can obtain the following difference apart from the constant coefficient.

(i) For $n=1$, the convolution kernel in the expression of $P^{*}Pf$ is $|x_1-u_1|^{-1}$ instead of $\|(x_1,u_1)\|^{-1}$ in $R^{\#}Rf$. There is a noticeable difference between the two kernels. Note that
$$
\|(x_1,u_1)\|^{-1} \to 0, \quad {\rm as} \ x_1\to \infty \ {\rm and} \ u_1 \to \infty.
$$
Nevertheless, this is not true for $|x_1-u_1|^{-1}$ since letting $x_1=u_1+1$ we have
$$
|x_1-u_1|\equiv1 \nrightarrow 0,\quad {\rm as} \ x_1\to \infty \ {\rm and} \ u_1 \to \infty.
$$

(ii) For $n=2$, the difference between two convolution kernels is more significant. First of all, the convolution kernel in $P^{*}Pf$ includes $\delta$-function which depends on the angle variables $\phi$ and $\varphi$. In other words, the convolutional representation only becomes non-zero under the specific condition when $\phi=\varphi$ or $\phi=\varphi\pm \pi$. But this holds without any condition for the convolution representation of $R^{\#}Rf$. In addition, the norm part of the convolution kernel of $P^{*}Pf$ is $\|(x_1,u_1)\|^{-1} \|(x_2,u_2)\|^{-1}$ instead of $\|(x_1,x_2,u_1,u_2)\|^{-1}$ in $R^{\#}Rf$. In fact, by Cauchy-Schwarz inequality we have
$$
\|(x_1,u_1)\|^{-1} \|(x_2,u_2)\|^{-1} \le \|(x_1,x_2,u_1,u_2)\|^{-1}, \quad \forall \ x_1,x_2,u_1,u_2 \in \mathbb{R}.
$$
The above different features of the convolution kernel in $P^{*}Pf$ are all generated by the coupling relation between variables in the definition of the photography transform.
} 
\end{remark}

\section{\bf Inversion formula for the photography transform}\label{se:inversion}

An exact analytic inversion formula for the photography transform is not only of great importance for the development of inversion algorithms, especially for analytic reconstruction methods, but also plays a vital role in the study of the local dependence of the solution on the data. In this section, we derive the analytic inversion formulas for the $1$-dimensional photography transform without any assumption and for the $2$-dimensional photography transform with a certain assumption.

We first recall the Fourier slice theorem in the last section for $n=1$, that is,
\begin{equation}
(Pf) \,\hat{}\, (\alpha, \xi_{\bar{x}_1}) = \hat{f}(\alpha \xi_{\bar{x}_1},(1-\alpha) \xi_{\bar{x}_1}), \quad \alpha \in \mathbb{R}.\label{Pn1}
\end{equation}
According to our analysis in Remark \ref{R3.5}, all values of $\hat{f}(\xi_{x_1},\xi_{u_1})$ in the frequency domain can be uniquely determined by $(Pf) \,\hat{}\, (\alpha, \xi_{\bar{x}_1})$ which means $\hat{f}$ is known on $\mathbb{R}^2$ in this sense (see Figure \ref{n1}). This analysis can be used as a starting point for constructing the inversion formula for the $1$-dimensional photography transform. 

For $n=2$ we have 
\begin{equation}
(Pf) \,\hat{}\, (\alpha, \xi_{\bar{x}_1},\xi_{\bar{x}_2}) = \hat{f}(\alpha \xi_{\bar{x}_1},\alpha \xi_{\bar{x}_2},(1-\alpha) \xi_{\bar{x}_1},(1-\alpha) \xi_{\bar{x}_2}), \quad \alpha \in \mathbb{R}.\label{Pn2}
\end{equation}
Unfortunately, all values of $\hat{f}(\xi_{{x}_1},\xi_{{x}_2},\xi_{{u}_1},\xi_{{u}_2})$ in the frequency domain can not be determined by $(Pf) \,\hat{}\, (\alpha, \xi_{\bar{x}_1},\xi_{\bar{x}_2})$, since one only knows the information of $\hat{f}$ on the 1-dimensional collection of 2-dimensional subspaces
\begin{equation}
\bigcup_{\alpha \in \mathbb{R}}\left\{(\xi_{{x}_1},\xi_{{x}_2},\xi_{{u}_1},\xi_{{u}_2})\in\mathbb{R}^4:\xi_{{x}_i}=\alpha\xi_{\bar{x}_i}, \ \xi_{{u}_i}=(1-\alpha) \xi_{\bar{x}_i}, \ \xi_{\bar{x}_i}\in\mathbb{R},\ i=1,2\right\}, \label{4.1.1}
\end{equation}
which, in fact, is only a subset of the whole frequency domain. To overcome this difficulty so that the 2-dimensional photography transform can be inverted, we introduce the following assumption.

\begin{figure}[htbp]
	\centering
    \begin{subfigure}
		\centering
		\begin{tikzpicture}[scale=1.14]
		\draw[->][line width =1.5pt]  (-2,1) -- (2,1);
		\draw[->][line width =1.5pt]  (-2,-1) -- (2,-1);
		\draw[->][line width =1.5pt]  (-2.4,0.7) -- (-2.4,-0.7);
		\draw[<-][line width =1.5pt]  (2.4,0.7) -- (2.4,-0.7);
		\node at(0,1.2) {\small $2$-D Fourier Transform};
		\node at(0,-0.8) {\small $1$-D Fourier Transform};
		\node at(-2.7,0) {$P$};
		\node at(2.4,1) {$\hat{f}$};
		\node at(-2.4,-1) {$Pf$};
		\node at(2.5,-1) {$(Pf)\, \hat{} \, {}$};
       \node at(-2.4,1) {$f$};
		\end{tikzpicture}
		\caption{The relationship between the Fourier transform and the photography transform $P$ when $n=1$.}
		\label{n1}
	\end{subfigure}
	\begin{subfigure}
		\centering
		\begin{tikzpicture}[scale=1.14]
		\draw[->][line width =1.5pt]  (-2,1) -- (2,1);
		\draw[->][line width =1.5pt]  (-2,-1) -- (2,-1);
		\draw[->][line width =1.5pt]  (-2.4,0.7) -- (-2.4,-0.7);
		\draw[<-][line width =1.5pt]  (2.4,0.7) -- (2.4,-0.7);
		\node at(0,1.2) {\small $4$-D Fourier Transform};
		\node at(0,-0.8) {\small $2$-D Fourier Transform};
		\node at(-2.7,0) {$P$};
		\node at(2.4,1) {$\hat{f}$};
		\node at(-2.4,-1) {$Pf$};
		\node at(2.5,-1) {$(Pf)\, \hat{} \ {}$};
       \node at(-2.4,1) {$f$};
       \node at(1.3,0) {Assumption \ref{assum1}};
		\end{tikzpicture}
		\caption{The relationship between the Fourier transform and the photography transform $P$ when $n=2$.}
		\label{n2}
	\end{subfigure}
	\centering
\end{figure}

\begin{assumption}\label{assum1}
For $f(\boldsymbol{x},\boldsymbol{u}) \in \mathcal{S}$$(\mathbb{R}^{4})$, the Fourier transform of $f$ satisfies

\begin{equation}
\hat{f}( \xi_{x_1},\xi_{x_2}, s \xi_{x_1} + t  \xi_{x_2},s \xi_{x_2} - t \xi_{x_1}) 
=\hat{f}(\xi_{x_1},\xi_{x_2}, s \xi_{x_1} , s \xi_{x_2})\delta(t), \quad \forall \ \xi_{x_1},\xi_{x_2},s,t \in \mathbb{R}. \label{4.1.2}
\end{equation}
\end{assumption}

\begin{remark}{\rm
If $f$ satisfies the Assumption \ref{assum1}, its Fourier transform includes non-zero entries only on the 1-dimensional collection of 2-dimensional subspaces
\begin{equation}\label{4.1.3}
\bigcup_{s \in \mathbb{R}}\left\{(\xi_{{x}_1},\xi_{{x}_2},\xi_{{u}_1},\xi_{{u}_2})\in\mathbb{R}^4:\xi_{{u}_i}=s\xi_{x_i} ,i=1,2\right\}.
\end{equation}
It is noted that the light field in a Lambertian scene of slope $s$ \cite{LA,PF,TK} in which all rays from a point have the same radiance, satisfies this assumption.}
\end{remark}

Then, for $n=2$, under the above assumption, we can obtain all values of $\hat{f}$ on the whole frequency domain according to $(Pf) \,\hat{} \,$ (see Figure \ref {n2}), since (\ref{4.1.1}) coincides with (\ref{4.1.3}) by making appropriate variable substitutions. Thus, we can also establish the inversion formula for the 2-dimensional photography transform.

To proceed, we need to introduce the Riesz potential $I^{\beta}$ and the coupled Riesz potential $I^{\beta}_{c}$ by 
$$
(I^{\beta}f)\, \hat{} \, (\xi_{\boldsymbol{x}},\xi_{\boldsymbol{u}}) = \|(\xi_{\boldsymbol{x}},\xi_{\boldsymbol{u}}) \|^{-\beta} \hat{f}(\xi_{\boldsymbol{x}},\xi_{\boldsymbol{u}}), \quad \beta<2n,
$$
and
$$
(I^{\beta}_{c}f)\, \hat{} \, (\xi_{x_1},\xi_{x_2},\xi_{u_1},\xi_{u_2}) = \|(\xi_{x_1},\xi_{x_2})\|^{-\beta}\| (\xi_{u_1},\xi_{u_2})  \|^{-\beta} \hat{f}(\xi_{x_1},\xi_{x_2},\xi_{u_1},\xi_{u_2}),\quad \beta<2.
$$
If $I^{\beta}$ is applied to functions $g(\alpha,\bar{\boldsymbol{x}})$ on $\mathbb{R}\times\mathbb{R}^{n}$ it acts on the second variable. By the definition of $I^{\beta}$ and $I^{\beta}_{c}$, the following formulas 
\begin{gather*}
I^{-\beta}I^{\beta}f = f,\\
I^{-\beta}_{c}I^{\beta}_c f = f,
\end{gather*}
can be verified easily.

Now, we give the analytic inversion formulas for the $1$- and $2$-dimensional photography transform $P$, respectively.

\begin{theorem}
Let $f(\boldsymbol{x},\boldsymbol{u})\in\mathcal{S}$$(\mathbb{R}^{2n})$ and denote
\begin{equation*}
P^{*}_{\beta}g = \begin{cases}
P^{*}\Big(\big(\alpha^2+(1-\alpha)^2\big)^{-\frac{\beta}{2}}g \Big),  & \text{if} \ n=1,\\
P^{*}\big(|\alpha|^{2-\beta} |1-\alpha|^{-\beta} g\big), & \text{if}\ n=2.\\
\end{cases}
\end{equation*}
Then  
\\{\rm (}i{\rm )} for $n=1$ and any $\beta <2$, we have
\begin{equation}
f=I^{-\beta}P^{*}_{\beta}I^{\beta -1}g,\quad g=Pf;\label{Invn1}
\end{equation}
{\rm (}ii{\rm )} for $n=2$ and any $\beta <2$, under Assumption \ref{assum1}, we have
\begin{equation}
f = I^{-\beta}_{c} P^{*}_{\beta}I^{2(\beta -1)}g,\quad g=Pf.\label{Invn2}
\end{equation}
\end{theorem}

\begin{proof}
We first prove the result for $n=1$. By the definition of $I^{\beta}$, we get
$$
I^{\beta}f(x_1,u_1) = \int_{\mathbb{R}^2} \hat{f}(\xi_{x_1},\xi_{u_1})\|(\xi_{x_1},\xi_{u_1})\|^{-\beta} e^{2\pi i (\xi_{x_1}x_1+\xi_{u_1}u_1)} \dif \xi_{x_1}\xi_{u_1}.
$$
Changing variables by letting $\xi_{x_1}=\alpha\xi_{\bar{x}_1}$ and $\xi_{u_1} = (1-\alpha)\xi_{\bar{x}_1}$ yields
\begin{equation*}
I^{\beta} f(x_1,u_1) =  \int_{\mathbb{R}}\int_{\mathbb{R}}\hat{f}(\alpha\xi_{\bar{x}_1},(1-\alpha)\xi_{\bar{x}_1})\big(\alpha^2+(1-\alpha)^2\big)^{-\frac{\beta}{2}}|\xi_{\bar{x}_1}|^{1-\beta} e^{ 2\pi i A\xi_{\bar{x}_1}}\dif\xi_{\bar{x}_1} \dif \alpha,
\end{equation*}
where $A=\alpha x_1+(1-\alpha)u_1$.
Using (\ref{Pn1}) we obtain
\begin{align*}
I^{\beta}f(x_1,u_1) =  \int_{\mathbb{R}}\big(\alpha^2+(1-\alpha)^2\big)^{-\frac{\beta}{2}} \int_{\mathbb{R}}(Pf)\, \hat{} \,(\alpha,\xi_{\bar{x}_1})|\xi_{\bar{x}_1}|^{1-\beta}
 e^{2\pi i A \xi_{\bar{x}_1}}\dif\xi_{\bar{x}_1}\dif \alpha.
\end{align*}
The inner integral can be expressed as a form with Riesz potential, hence
\begin{align*}
I^{\beta}f(x_1,u_1) 
&= \int_{\mathbb{R}}\big(\alpha^2+(1-\alpha)^2\big)^{-\frac{\beta}{2}}I^{1-\beta}Pf(\alpha,\alpha x_1+(1-\alpha)u_1)\dif \alpha \\
&= P^{*}_{\beta}I^{1-\beta}Pf(x_1,u_1)
\end{align*}
and then the inversion formula for $n=1$ follows by applying $I^{-\beta}$ to both sides of the above equality.

Similarly, for $n=2$ we have 
\begin{align*}
I^{\beta}_c f(x_1,x_2,u_1,u_2)
&=\int_{\mathbb{R}^2}\int_{\mathbb{R}^2}\hat{f}(\xi_{x_1},\xi_{x_2},\xi_{u_1},\xi_{u_2})\|(\xi_{x_1},\xi_{x_2})\|^{-\beta}\|(\xi_{u_1},\xi_{u_2})\|^{-\beta}\\\
&\qquad\times e^{2\pi i (\xi_{x_1}x_1+\xi_{x_2}x_2+\xi_{u_1}u_1+\xi_{u_2}u_2)}\dif\xi_{u_1}\xi_{u_2}\dif \xi_{x_1}\xi_{x_2}.
\end{align*}
Introducing the substitutions $\xi_{u_1}=s \xi_{x_1} + t  \xi_{x_2}$, $\xi_{u_2}=s \xi_{x_2} - t \xi_{x_1}$ in the inner integral yields
\begin{align*}
I^{\beta}_cf(x_1,x_2,u_1,u_2)
&=\int_{\mathbb{R}^2}\int_{\mathbb{R}^2}\hat{f}(\xi_{x_1}, \xi_{x_2},s \xi_{x_1} + t  \xi_{x_2},s \xi_{x_2} - t \xi_{x_1})\|(\xi_{x_1},\xi_{x_2}) \|^{2-\beta} \\
&\qquad\times \|(s \xi_{x_1} + t  \xi_{x_2},s \xi_{x_2} - t \xi_{x_1}) \|^{-\beta} e^{2\pi i B} \dif s t \dif \xi_{x_1}\xi_{x_2},
\end{align*}
where 
$$
B=(x_1+s u_1- t u_2)\xi_{x_1} + (x_2+ s u_2 + t u_1)\xi_{x_2}.
$$
By Assumption \ref{assum1},
\begin{align*}
I^{\beta}_cf(x_1,x_2,u_1,u_2)
&=\int_{\mathbb{R}^2}\int_{\mathbb{R}}\int_{\mathbb{R}}\hat{f}(\xi_{x_1}, \xi_{x_2},s \xi_{x_1},s \xi_{x_2}) \delta(t)  \|(\xi_{x_1},\xi_{x_2}) \|^{2-\beta}\\
&\qquad \times  \|(s \xi_{x_1} + t  \xi_{x_2},s \xi_{x_2} - t \xi_{x_1}) \|^{-\beta}e^{2\pi i B} \dif t\dif s \dif \xi_{x_1}\xi_{x_2}.
\end{align*}
Using the properties of the $\delta$-function follows that
\begin{align*}
I^{\beta}_c f(x_1,x_2,u_1,u_2)
&=\int_{\mathbb{R}^2}\int_{\mathbb{R}} \hat{f}(\xi_{x_1}, \xi_{x_2},s \xi_{x_1},s \xi_{x_2}) \|(\xi_{x_1},\xi_{x_2}) \|^{2(1-\beta)}|s|^{-\beta}\\
& \qquad \times e^{2\pi i [\xi_{x_1}(x_1+s u_1)+\xi_{x_2}(x_2+ s u_2 )]} \dif s \dif \xi_{x_1}\xi_{x_2}.
\end{align*}
Making the substitutions $s = \frac{1}{\alpha}-1 $, $\xi_{x_1}=\alpha \xi_{\bar{x}_1}$ and $\xi_{x_2}=\alpha \xi_{\bar{x}_2}$, we have
\begin{align*}
I^{\beta}_c f(x_1,x_2,u_1,u_2)
&=\int_{\mathbb{R}}\int_{\mathbb{R}^2}\hat{f}(\alpha\xi_{\bar{x}_1},\alpha\xi_{\bar{x}_2},(1-\alpha)\xi_{\bar{x}_1},(1-\alpha)\xi_{\bar{x}_2})\|(\xi_{\bar{x}_1},\xi_{\bar{x}_2}) \|^{2(1-\beta)}\\
&\qquad \times  |\alpha|^{2-\beta} |1-\alpha|^{-\beta} e^{2\pi i C}\dif \xi_{\bar{x}_1}\xi_{\bar{x}_2}\dif \alpha,
\end{align*}
where 
$$
C = (\alpha x_1+(1-\alpha) u_1) \xi_{\bar{x}_1} + ( \alpha x_2+(1-\alpha)u_2) \xi_{\bar{x}_2}.
$$
Here we express $\hat{f}$ by $(Pf)\, \hat{}\,$ in (\ref{Pn2}). It follows that
\begin{align*}
I^{\beta}_c f(x_1,x_2,u_1,u_2) &= \int_{\mathbb{R}}
 |\alpha|^{2-\beta} |1-\alpha|^{-\beta} \int_{\mathbb{R}^2} (Pf)\, \hat{}\,(\alpha,\xi_{\bar{x}_1},\xi_{\bar{x}_2}) \\
& \quad \quad \quad \times \|(\xi_{\bar{x}_1},\xi_{\bar{x}_2}) \|^{2(1-\beta)}  e^{2\pi i C}\dif \xi_{\bar{x}_1}\xi_{\bar{x}_2}\dif\alpha.
\end{align*}
By definition of the Riesz potential, hence
\begin{align*}
& I^{\beta}_c f(x_1,x_2,u_1,u_2)\\
{}=&\int_{\mathbb{R}} |\alpha|^{2-\beta}|1-\alpha|^{-\beta}I^{2(1-\beta)}Pf(\alpha,\alpha x_1+(1-\alpha)u_1,\alpha x_2+(1-\alpha)u_2))\dif \alpha \\
{}=&P_{\beta}^{*}I^{2(\beta-1)}Pf(x_1,x_2,u_1,u_2),
\end{align*}
and the inversion formula for $n=2$ follows by applying $I^{-\beta}_{c}$ to both sides of the above equality.
\end{proof}

\begin{remark}
\rm{
Let $f(\boldsymbol{x},\boldsymbol{u}) \in \mathcal{S}$$(\mathbb{R}^{2n})$. For any $\beta< 2n$, the inversion formula of the Radon transform is given by \cite{Na}
\begin{equation}
f=\frac{1}{2}I^{-\beta}R^{\#}I^{\beta-2n+1}g,\quad g=Rf. \label{RadonInversion}
\end{equation}
By comparing the inversion formula for the photography transform with that of the classic Radon transform, we can find the following facts.

(i) For $n=1$ and any $\beta<2$, the inversion formula (\ref{RadonInversion}) reads as
\[
f=\frac{1}{2}I^{-\beta}R^{\#}I^{\beta-1}g,\quad g=Rf.
\]
It is worth mentioning formally that the 2-dimensional classic Radon transform inversion formula is very similar to that of the photography transform in (\ref{Invn1}). However, there still exists a little difference in the part of the dual operator. To be more precise, $R^{\#}g$ is equivalent to applying $R^{\#}$ to $g$ directly. However, $P^{*}_{\beta}g$ is equivalent to applying $P^{*}$ to the product of $g$ and a weight function $\omega(\alpha)=\big(\alpha^2+(1-\alpha)^2\big)^{-\beta /2}$. 

(ii) For $n=2$ and any $\beta<4$, the inversion formula (\ref{RadonInversion}) reads as

\[
f=\frac{1}{2}I^{-\beta}R^{\#}I^{\beta-3}g,\quad g=Rf.
\]
Three differences need to be mentioned here regarding the inversion formulas of two transforms from right to left when $n=2$. The first one is the coefficient of the Riesz potential which acts on $g$ directly. The coefficient of the former is $2(\beta-1)$, and the latter one is $\beta-3$. The second one is the part of the dual operator, analogous to that when $n = 1$, but the weight function for $P^{*}_{\beta}g$ here is $\omega(\alpha) = |\alpha|^{2-\beta} |1-\alpha|^{-\beta}$. Finally, the coupled Riesz potential, which reflects the coupling relation between variables in the photography transform $P$, is different from the Riesz potential in the Radon transform inversion formula. This one illustrates the essential difference between the two transforms. For the $\delta$-function in the $2$-dimensional photography transform, the coefficients multiplied by $x_1$ and $x_2$ are the same as $\alpha$. Thus, $x_1$ and $x_2$ are coupled together in inversion, and so do $u_1$ and $u_2$. However, the classic Radon transform does not have such a coupling relation between variables.
}
\end{remark}

By taking $\beta=0$ and $\beta=1$ in the inversion formulas for the photography transform when $n=1$ and $n=2$, we can obtain the following two corollaries directly.

\begin{cor}
Letting $\beta=0$, 
\\{\rm (}i{\rm )} for $n=1$, we have
\begin{equation}
f=P^{*}I^{-1}g,\quad g=Pf;\label{Beta0n1}
\end{equation}
{\rm (}ii{\rm )} for $n=2$ and under Assumption \ref{assum1}, we have
\begin{equation}
f = P^{*}_{0}I^{-2}g,\quad g=Pf.\label{Beta0n2}
\end{equation}
\end{cor}

\begin{remark}\leavevmode
\rm{(i) For implementation of (\ref{Beta0n2}), we consider the approximate formula
$$
f\approx \int_{-\infty}^{+\infty} \alpha^2 (h_b*g)(\alpha,\alpha x_1+(1-\alpha u_1),\alpha x_2+(1-\alpha)u_2))\dif \alpha,
$$
where
$$
h_b(\bar{x}_1,\bar{x}_2)=\int_{\| (\xi_{\bar{x}_1},\xi_{\bar{x}_2})\|<b }\| (\xi_{\bar{x}_1},\xi_{\bar{x}_2}) \|^2 e^{2\pi i(\xi_{\bar{x}_1}\bar{x}_1+ \xi_{\bar{x}_2}\bar{x}_2)}\dif\xi_{\bar{x}_1}\xi_{\bar{x}_2},
$$
and $b$ is the cut-off frequency. This is the FBP method for the $2$-dimensional photography transform and other similar forms can be found in previous studies \cite{LA,LC}. The FBP method for the $1$-dimensional photography transform can also be obtained by a similar discussion.

(ii) By the Hilbert transform $H$ defined by
$$
(Hg)\, \hat{} \, (\xi_{\bar{x}_1}) = -i\, {\rm sgn}(\xi_{\bar{x}_1}) \hat{g}(\xi_{\bar{x}_1}),
$$
where 
\begin{equation*}
{\rm sgn}(\xi_{\bar{x}_1}) = \begin{cases}
1,  &\text{if} \ \xi_{\bar{x}_1}>0,\\
0, &\text{if}  \ \xi_{\bar{x}_1}=0,\\
-1, &\text{if}  \ \xi_{\bar{x}_1}<0,
\end{cases}
\end{equation*}
and the equality $I^{-1}g= (2\pi)^{-1} Hg^{(1)}$, we can obtain an equivalent formula of (\ref{Beta0n1})
\begin{align*}
f&= (2\pi)^{-1}P^{*}Hg^{(1)}\\
&= (2\pi)^{-1}\int_{\mathbb{R}} Hg^{(1)}(\alpha,\alpha x_1+(1-\alpha)u_1)\dif \alpha,
\end{align*}
where the derivative $g^{(1)}$ is taken with respect to the second variable. Unfortunately, the problem of reconstructing a function from its photography transform is not local in the sense that computing the function at some point requires all values of $g$ due to the non-locality of Hilbert transform \cite{Na}.

(iii) By the 2-dimensional Laplacian $\Delta$ and the equality $I^{-2}=-(2\pi)^{-2}\Delta$, an equivalent formula of (\ref{Beta0n2}) can be obtained as follows,
\begin{align*}
f&=-(2\pi)^{-2}P^{*}_{0}\Delta g \\
 &=-(2\pi)^{-2}\int_{\mathbb{R}}\alpha^2 (\Delta g)(\alpha,\alpha x_1+(1-\alpha)u_1,\alpha x_2+(1-\alpha)u_2)\dif \alpha,
\end{align*}
where $\Delta g$ is acting on the variables $\bar{x}_1$ and $\bar{x}_2$ of $g$. Since the Laplacian is a local operator, the inversion for the 2-dimensional photography transform only relies on local information of $g$.
}
\end{remark}

\begin{cor}
Letting $\beta=1$,
\\{\rm (}i{\rm )} for $n=1$, we have
\begin{equation}
f=I^{-1}P^{*}_{1}g,\quad g=Pf;\label{Beta1n1}
\end{equation}
{\rm (}ii{\rm )} for $n=2$, and under Assumption \ref{assum1}, we have
\begin{equation}
f = I^{-1}_{c} P^{*}_1 g,\quad g=Pf.\label{Beta1n2}
\end{equation}
\end{cor}

\begin{remark}{\rm
To consider the implementation of (\ref{Beta1n2}), we can get the BFP method of the $2$-dimensional photography transform $P$
$$
f\approx m_b*(P^{*}_1 g),
$$
where
\begin{align*}
m_b(x_1,x_2,u_1,u_2) 
&= \int_{\|(\xi_{x_1},\xi_{x_2})\|<b}\int_{\|(\xi_{u_1},\xi_{u_2})\|<b}\| (\xi_{x_1},\xi_{x_2})\| \| (\xi_{u_1},\xi_{u_2}) \| |\alpha| |1-\alpha|^{-1} \\  
& \qquad  \times  e^{2 \pi i(\xi_{x_1}x_1+\xi_{x_2}x_2+\xi_{u_1}u_1+\xi_{u_2}u_2)} \dif u_1 u_2\dif x_1 x_2.
\end{align*}
The BFP method of the $1$-dimensional photography transform $P$ can also be similarly obtained.}
\end{remark}

\section{\bf Some other extensions}\label{se:extension}
The main goals of this section are to propose and study two other types of photography transform. The proofs in this section are slight modifications of the proof in previous sections, so we omit them.

Suppose $f(\boldsymbol{x},\boldsymbol{u}) \in \mathcal{S}$$(\mathbb{R}^{2n})$. To consider that the variable $x_i$ in (\ref{2.2}) is multiplied by different $\alpha_i$ instead of the same $\alpha$, we obtain a new photography transform $\bar{P}$ defined by
\begin{align*}
\bar{P}f(\boldsymbol{\alpha},\boldsymbol{\bar{x}})
=\left( \prod_{i=1}^{n}|\alpha_i|\right)^{-1}\int_{\mathbb{R}^n} f\left(\bigg(\frac{1}{\alpha_i} \bar{x}_i+\bigg(1+\frac{1}{\alpha_i}\bigg)u_i\bigg)_{i=1}^{n},\boldsymbol{u}\right) \dif \boldsymbol{u},\label{5.1}
\end{align*}
for all $\boldsymbol{\alpha}=(\alpha_1,\ldots,\alpha_n)^{\top}\in (\mathbb{R}-\{0\})^n$ and $\boldsymbol{\bar{x}}\in\mathbb{R}^n$. For $n=1$, $\bar{P}$ and $P$ coincide. Based on the $\delta$-function, we get an equivalent form of $\bar{P}$
\[
\bar{P}f(\boldsymbol{\alpha},\boldsymbol{\bar{x}})=\int_{\mathbb{R}^n}\int_{\mathbb{R}^n}f(\boldsymbol{x},\boldsymbol{u}) \delta((\alpha_i x_i+(1-\alpha_i)u_i-\bar{x}_i)_{i=1}^n) \dif\boldsymbol{x} \dif\boldsymbol{u}.
\]
It can be seen from this equivalent definition that $x_i$ and $u_i$ in $\delta$-function are multiplied by different coefficients, so there is no coupling relation between variables.

Now, we introduce another new type of the photography transform $\widetilde{P}$. If $n$ is even, for all $\boldsymbol{\alpha}=(\alpha_1,\ldots,\alpha_{\frac{n}{2}})^{\top}\in (\mathbb{R}-\{0\})^{\frac{n}{2}}$ and $\boldsymbol{\bar{x}}\in\mathbb{R}^n$, $\widetilde{P}$ is defined by
\begin{align*}
\widetilde{P}f(\boldsymbol{\alpha},\boldsymbol{\bar{x}})
=\left( \prod_{i=1}^{n/2}|\alpha_i|^{2}\right)^{-1} \int_{\mathbb{R}^n}f\bigg(\bigg( \frac{1}{\alpha_i}\boldsymbol{\bar x}_i+\bigg(1+\frac{1}{\alpha_i}\boldsymbol{u}_i \bigg) \bigg)_{i=1}^{\frac{n}{2}}, \boldsymbol{u}\bigg)\dif\boldsymbol{u},\label{5.2}
\end{align*}
where $\boldsymbol{\bar x}_i=(\bar{x}_{2i-1},\bar{x}_{2i})^{\top}$. We give an equivalent definition
$$
\widetilde{P}f(\boldsymbol{\alpha},\boldsymbol{\bar{x}})
=\int_{\mathbb{R}^n}\int_{\mathbb{R}^n}f(\boldsymbol{x},\boldsymbol{u}) \delta((\alpha_i \boldsymbol{x}_i+(1-\alpha_i)\boldsymbol{u}_i-\boldsymbol{\bar{x}}_i)_{i=1}^{\frac{n}{2}}) \dif \boldsymbol{x} \dif\boldsymbol{u}.
$$
If $n$ is odd, for all $\boldsymbol{\alpha}=(\alpha_1,\ldots,\alpha_{\frac{n+1}{2}})^{\top}\in (\mathbb{R}-\{0\})^{\frac{n+1}{2}}$ and $\boldsymbol{\bar{x}}\in\mathbb{R}^n$, $\widetilde{P}$ is defined by 
\begin{align*}
\widetilde{P}f(\boldsymbol{\alpha},\boldsymbol{\bar{x}})
&=\left( |\alpha_n|\prod_{i=1}^{(n-1)/2}|\alpha_i|^{2}\right)^{-1} \int_{\mathbb{R}^n}f\bigg(\bigg( \frac{1}{\alpha_i}\boldsymbol{\bar x}_i+\bigg(1+\frac{1}{\alpha_i}\boldsymbol{u}_i \bigg) \bigg)_{i=1}^{\frac{n-1}{2}},\\
&\quad \quad \quad \frac{1}{\alpha_{\frac{n+1}{2}}}\bar{x}_n+\bigg( 1+\frac{1}{\alpha_{\frac{n+1}{2}}} \bigg)u_n,  \boldsymbol{u}\bigg)\dif\boldsymbol{u},
\end{align*}
and its equivalent definition is given by
\begin{align*}
\widetilde{P} f(\boldsymbol{\alpha},\boldsymbol{\bar{x}}) 
&= \int_{\mathbb{R}^n}\int_{\mathbb{R}^n} f(\boldsymbol{x},\boldsymbol{u}) \\
&\times \delta \left( (\alpha_i \boldsymbol{x}_i+(1-\alpha_i)\boldsymbol{u}_i-\boldsymbol{\bar{x}}_i )_{i=1}^{\frac{n-1}{2}}, \alpha_{\frac{n+1}{2}}x_n+(1-\alpha_{\frac{n+1}{2}})u_n-\bar{x}_n\right) \dif \boldsymbol{x} \dif\boldsymbol{u}.    
\end{align*}
For $n=1,2$, $\widetilde{P}$ and $P$ coincide. When $n$ is even, there is a coupling relation between $x_{2i-1}$ and $x_{2i}$ in $(\boldsymbol{x}_i,\boldsymbol{u}_i)$. Meanwhile, when $n$ is odd, the variables have a coupling relation similar to the even case except for $(x_n,u_n)$. We also write $\bar{P}_{\boldsymbol{\alpha}}(\boldsymbol{\bar{x}}) =\bar{P}f(\boldsymbol{\alpha},\boldsymbol{\bar{x}})$
and $\widetilde{P}_{\boldsymbol{\alpha}}(\boldsymbol{\bar{x}}) =\widetilde{P}f(\boldsymbol{\alpha},\boldsymbol{\bar{x}}).$

Now, the Fourier slice theorem and the convolution theorem of both $\bar{P}_{\boldsymbol{\alpha}}$ and $\widetilde{P}_{\boldsymbol{\alpha}}$ are given as follows.

\begin{proposition}
For $f(\boldsymbol{x},\boldsymbol{u}),\ g(\boldsymbol{x},\boldsymbol{u})\in\mathcal{S}$$(\mathbb{R}^{2n})$, we have
\\{\rm (}i{\rm )}
$$
(\bar{P}_{\boldsymbol{\alpha}}f)\, \hat{} \, (\xi_{\bar{\boldsymbol{x}}}) = \hat{f}((\alpha_i \xi_{\bar{x}_i})_{i=1}^{n},((1-\alpha_i) \xi_{\bar{x}_i})_{i=1}^n), \quad \boldsymbol{\alpha} \in \mathbb{R}^n, 
$$
if $n$ is even 
$$
(\widetilde{P}_{\boldsymbol{\alpha}}f)\, \hat{} \,(\xi_{\bar{\boldsymbol{x}}})=\hat{f}((\alpha_i \xi_{\boldsymbol{\bar{x}}_i})_{i=1}^{\frac{n}{2}},((1-\alpha_i) \xi_{\boldsymbol{\bar{x}}_i})_{i=1}^{\frac{n}{2}}), \quad \boldsymbol{\alpha} \in \mathbb{R}^\frac{n}{2},
$$
and if $n$ is odd,
$$
(\widetilde{P}_{\boldsymbol{\alpha}}f)\, \hat{} \, (\xi_{\bar{\boldsymbol{x}}})=\hat{f}((\alpha_i \xi_{\boldsymbol{\bar{x}}_i})_{i=1}^{\frac{n-1}{2}},\alpha_{\frac{n+1}{2}}\xi_{\bar{x}_n},((1-\alpha_i) \xi_{\boldsymbol{\bar{x}}_i})_{i=1}^{\frac{n-1}{2}},(1-\alpha_{\frac{n+1}{2}})\xi_{\bar{x}_n}), \quad \boldsymbol{\alpha} \in \mathbb{R}^\frac{n+1}{2};
$$
{\rm (}ii{\rm )} 
$$
\bar{P}_{\boldsymbol{\alpha}}(f*g) = \bar{P}_{\boldsymbol{\alpha}}f*\bar{P}_{\boldsymbol{\alpha}}g,\quad \boldsymbol{\alpha} \in \mathbb{R}^n,
$$
$$
\widetilde{P}_{\boldsymbol{\alpha}}(f*g) = \widetilde{P}_{\boldsymbol{\alpha}}f*\widetilde{P}_{\boldsymbol{\alpha}}g,
$$
where $\boldsymbol{\alpha} \in \mathbb{R}^\frac{n}{2}$ if $n$ is even and $\boldsymbol{\alpha} \in \mathbb{R}^\frac{n+1}{2}$ if $n$ is odd.
\end{proposition}

Along the similar argument with (\ref{3.3.1})$-$(\ref{3.3.3}), the dual operators $\bar{P}^{*}$ and $\widetilde{P}^{*}$ are defined by
$$
(\bar{P}^{*}g)(\boldsymbol{x},\boldsymbol{u})=\int_{\mathbb{R}^n} g\left(\boldsymbol{\alpha},(\alpha_i x_i+(1-\alpha_i)u_i)_{i=1}^n\right)\dif\boldsymbol{\alpha},
$$
if $n$ is even
$$
(\widetilde{P}^{*}g)(\boldsymbol{x},\boldsymbol{u})=\int_{\mathbb{R}^{\frac{n}{2}}} g\left(\boldsymbol{\alpha},(\alpha_i \boldsymbol{x}_i+(1-\alpha_i)\boldsymbol{u}_i\right)_{i=1}^{\frac{n}{2}})\dif\boldsymbol{\alpha},
$$
and if $n$ is odd
$$
(\widetilde{P}^{*}g)(\boldsymbol{x},\boldsymbol{u})=\int_{\mathbb{R}^{\frac{n+1}{2}}}g\left(\boldsymbol{\alpha},(\alpha_i \boldsymbol{x}_i+(1-\alpha_i)\boldsymbol{u}_i)_{i=1}^{\frac{n-1}{2}},\alpha_{\frac{n+1}{2}} x_n+(1-\alpha_{\frac{n+1}{2}})u_n\right)\dif\boldsymbol{\alpha}.
$$
Similarly, we can derive the convolution property related to the dual operators and the representation of the normal operator for $\bar P$ and $\widetilde{P}$.
\begin{proposition} 
For $f(\boldsymbol{x},\boldsymbol{u}) \in \mathcal{S}$$(\mathbb{R}^{2n})$ and
\\{\rm (}i{\rm )} $g(\boldsymbol{\alpha},\boldsymbol{\bar{x}})\in\mathcal{S}$$(\mathbb{R}^{n}\times \mathbb{R}^{n})$, we have 
$$
(\bar{P}^{*}g)*f = \bar{P}^{*}(g*Pf);
$$
{\rm (}ii{\rm )} if $n$ is even $g(\boldsymbol{\alpha},\boldsymbol{\bar{x}}) \in \mathcal{S}$$(\mathbb{R}^{\frac{n}{2}}\times \mathbb{R}^{n})$ and if $n$ is odd $g(\boldsymbol{\alpha},\boldsymbol{\bar{x}}) \in \mathcal{S}$$(\mathbb{R}^{\frac{n+1}{2}}\times \mathbb{R}^{n})$, we have 
$$
(\widetilde{P}^{*}g)*f = \widetilde{P}^{*}(g*Pf).
$$
\end{proposition}

\begin{proposition}
For $f(\boldsymbol{x},\boldsymbol{u}) \in \mathcal{S}$$(\mathbb{R}^{2n})$ we have
\\{\rm (}i{\rm )}
$$
\bar{P}^{*}\bar{P}f = \left\{\prod_{i=1}^{n}|x_i-u_i|^{-1} \right\}*f,
$$
{\rm (}ii{\rm )} if $n$ is even 
$$
\widetilde{P}^{*}\widetilde{P}f = \left\{ \frac{\prod_{i=1}^{n/2}\delta\left( \sin (\phi_i-\varphi_i)\right)}{\prod_{i=1}^{n/2}\|(x_i,u_i) \|} \right\}*f,
$$
where $\phi_i={\rm arg}(x_{2i-1},u_{2i-1})$, $\varphi_i={\rm arg}(x_{2i},u_{2i})$, $i=1,\ldots,n/2$ and if $n$ is odd, 
$$
\widetilde{P}^{*}\widetilde{P}f = \left\{ \frac{\prod_{i=1}^{(n-1)/2}\delta\big( \sin (\phi_i-\varphi_i)\big)}{|x_n-u_n|\prod_{i=1}^{(n-1)/2}\|(x_i,u_i) \|} \right\}*f.
$$
where $\phi_i={\rm arg}(x_{2i-1},u_{2i-1})$, $\varphi_i={\rm arg}(x_{2i},u_{2i})$, $i=1,\ldots,(n-1)/2$.
\end{proposition}

To proceed, we need to define two operators
$$
(\bar{I}^{\beta}_{c}f)\, \hat{} \, (\xi_{\boldsymbol{x}},\xi_{\boldsymbol{u}}) = \prod_{i=1}^{n} \|(\xi_{x_i},\xi_{u_i}) \|^{-\beta} \hat{f}(\xi_{\boldsymbol{x}},\xi_{\boldsymbol{u}}), \quad \beta<2,
$$
and
$$
(\bar{I}^{\beta}g)\, \hat{} \,(\boldsymbol{\alpha},\xi_{\boldsymbol{\bar{x}}}) = \prod_{i=1}^{n}|\xi_{\bar{x}_i}|^{-\beta}  \hat{g}(\boldsymbol{\alpha},\xi_{\boldsymbol{\bar{x}}}),\quad \beta<2.
$$
Then an exact analytic inversion formula for $\bar{P}$ is given by the next proposition.
\begin{proposition}
Let $f(\boldsymbol{x},\boldsymbol{u})\in\mathcal{S}$$(\mathbb{R}^{2n})$ and denote
$$
\bar{P}^{*}_{\beta}g = P^{*}\Big(\prod_{i=1}^{n} \big(\alpha_i^2+(1-\alpha_i)^2\big)^{-\frac{\beta}{2}}g\Big).
$$
For any $\beta<2$ we have 
$$
f=\bar{I}^{-\beta}_{c} \bar{P}^{*}_{\beta} \bar{I}^{\beta-1}g, \quad g=\bar{P}f.
$$
\end{proposition}

Since there exists the coupling relation between the 
variables in $\widetilde{P}$, we use the following assumption to overcome the inadequacy of information in the frequency domain when $\widetilde{P}$ is inverted.

\begin{assumption}\label{assum2}
For $f(\boldsymbol{x},\boldsymbol{u})\in\mathcal{S}$$(\mathbb{R}^{2n})$. 
{\rm (}i{\rm )} if $n$ is even, the Fourier transform of $f$ satisfies
$$
\hat{f} \left(\xi_{\boldsymbol{x}},(s_i\xi_{x_{2i-1}}+t_i\xi_{x_{2i}},s_i\xi_{x_{2i}}-t_i\xi_{x_{2i-1}})_{i=1}^{\frac{n}{2}}\right) = \hat{f}\left(\xi_{\boldsymbol{x}},(s_i\xi_{\boldsymbol{x}_i})_{i=1}^{\frac{n}{2}}\right) \delta(\boldsymbol{t}), 
$$
for all $\boldsymbol{t}=(t_1,\ldots,t_{\frac{n}{2}})^{\top} \in \mathbb{R}^\frac{n}{2}$, $\boldsymbol{s}=(s_1,\ldots,s_{\frac{n}{2}})^{\top} \in \mathbb{R}^\frac{n}{2}$, and $\xi_{\boldsymbol{x}}\in \mathbb{R}^n$;\\
{\rm (}ii{\rm )} if $n$ is odd, the Fourier transform of $f$ satisfies
$$
\hat{f}\left(\xi_{\boldsymbol{x}},(s_i\xi_{x_{2i-1}}+t_i\xi_{x_{2i}},s_i\xi_{x_{2i}}-t_i\xi_{x_{2i-1}})_{i=1}^{\frac{n-1}{2}},\xi_{u_n}\right) = \hat{f}\left(\xi_{\boldsymbol{x}},(s_i\xi_{\boldsymbol{x}_i})_{i=1}^{\frac{n-1}{2}},\xi_{u_n})\right)\delta(\boldsymbol{t}), 
$$
for all $\boldsymbol{t}=(t_1,\ldots,t_{\frac{n-1}{2}})^{\top} \in \mathbb{R}^\frac{n-1}{2}$, $\boldsymbol{s}=(s_1,\ldots,s_{\frac{n-1}{2}})^{\top} \in \mathbb{R}^\frac{n-1}{2}$,  $\xi_{\boldsymbol{x}}\in \mathbb{R}^n$, and $\xi_{u_n}\in \mathbb{R}$.
\end{assumption}

For any $\beta<2$, we define
\begin{equation*}
(\tilde{I}^{\beta}_{c}f)\, \hat{} \, (\xi_{\boldsymbol{x}},\xi_{\boldsymbol{u}}) = 
\begin{cases}
\prod_{i=1}^{n/2}\| \xi_{\boldsymbol{x}_i}\|^{-\beta}\| \xi_{\boldsymbol{u}_i}\|^{-\beta}\hat{f}(\xi_{\boldsymbol{x}},\xi_{\boldsymbol{u}}),  & \text{if} \ n \ \text{is even},\\
\prod_{i=1}^{(n-1)/2}\| \xi_{\boldsymbol{x}_i}\|^{-\beta}\| \xi_{\boldsymbol{u}_i}\|^{-\beta}\|(\xi_{x_n},\xi_{u_n}) \|^{-\beta}\hat{f}(\xi_{\boldsymbol{x}},\xi_{\boldsymbol{u}}), & \text{if} \ n \ \text{is odd},\\
\end{cases}
\end{equation*}
and
\begin{equation*}
(\tilde{I}^{\beta}g)\, \hat{} \, (\boldsymbol{\alpha},\xi_{\boldsymbol{\bar{x}}}) = 
\begin{cases}
\prod_{i=1}^{n/2}\| \xi_{\boldsymbol{\bar{x}}_i} \|^{-\beta}\hat{g}(\boldsymbol{\alpha},\xi_{\boldsymbol{\bar{x}}}), & \text{if} \ n \ \text{is even},\\
\prod_{i=1}^{(n-1)/2}\| \xi_{\boldsymbol{\bar{x}}_i} \|^{-\beta} |\xi_{\bar{x}_n}|^{-\frac{\beta}{2}} \hat{g}(\boldsymbol{\alpha},\xi_{\boldsymbol{\bar{x}}}), & \text{if} \ n \ \text{is odd}.\\
\end{cases}
\end{equation*}
Finally, we give the inversion formula of $\widetilde{P}$.

\begin{proposition}
Let $f(\boldsymbol{x},\boldsymbol{u})\in\mathcal{S}$$(\mathbb{R}^{2n})$ and Assumption \ref{assum2} hold. Denote 
\begin{equation*}
\widetilde{P}^{*}_{\beta}g= \begin{cases}
\widetilde{P}^{*}\Big(\prod_{i=1}^{n/2} |\alpha_i|^{2-\beta}|1-\alpha_i|^{-\beta}g\Big), & \text{if} \ n \ \text{is even},\\
\widetilde{P}^{*}\left(\prod_{i=1}^{\frac{n-1}{2}} |\alpha_i|^{2-\beta}|1-\alpha_i|^{-\beta}\left(\alpha_{\frac{n+1}{2}}^2+(1-\alpha_{\frac{n+1}{2}}^2)^2\right)^{-\frac{\beta}{2}}g\right), & \text{if} \ n \ \text{is odd}.
\end{cases}
\end{equation*}
For any $\beta<2$, we have
\[
f = \tilde{I}^{-\beta}_c \widetilde{P}^{*}_{\beta} \tilde{I}^{2(\beta-1)}g, \quad g=\widetilde{P}f.
\]
\end{proposition}

\section{\bf Conclusion and remarks}

In this paper, three types of photography transform $P$, $\bar{P}$, and $\widetilde{P}$ are defined from an application in the real world, that is, to reconstruct the light field from the focal stack. We clarify the equivalent relation between these photography transforms and the Radon transform by introducing some coupled Radon transforms. Several properties of the photography transforms are obtained, which are related to the Fourier transform, convolution, and their dual operator. More importantly, the analytic inversion formulas of these photography transforms are also given under appropriate assumptions. All of the above results show the similarities and differences between the photography transform and the classic Radon transform. 

Some additional remarks:
\begin{itemize}
\item For different photography transforms we need to note that $P$ has a stronger coupling relation between variables than $\widetilde{P}$, while $\bar{P}$ dose not have coupling relation between variables. The stronger coupling relation variables in integral means the less dimensional measurement will be, which makes the inversion more difficult. In this paper, although the inversion formulas of $\bar{P}$ and $\widetilde{P}$ for the arbitrary $n$-dimensional case have been obtained, we only derive the analytic inversion formulas for $P$ when $n=1,2$.

\item This paper mainly aims to establish a theoretical analysis of the photography transform $P$, so it is assumed that the measurement $Pf$ is known for any real number $\alpha$. In practice, we can only sample a few non-negative discrete points of $\alpha$ for $Pf$ when $n=2$, so the use of exact inversion formula (\ref{Invn2}) will face the difficulty of the incomplete data similar to limited-angle and sparse-view reconstruction in CT \cite{Si}. How to construct a stable and efficient numerical algorithm for the light filed reconstruction from the insufficient measurement of the focal stack is a big challenge in practice.
\end{itemize}
\section*{\bf Acknowledgement}
\noindent
This work was partially supported by the National Natural Science Foundation of China (Grant No. 61931003), and partially supported by the China Scholarship Council (Grant No. 202307090070). The authors would like to thank T. Shan for helpful comments and suggestions.

\end{document}